\documentclass[12pt,a4paper]{scrartcl}

\usepackage[utf8]{inputenc}
\usepackage[british]{babel}
\usepackage{amsmath,amsthm,amssymb,amsfonts}
\usepackage{mathtools}
\mathtoolsset{centercolon}
\usepackage{hyperref}
\usepackage{float,caption}
\usepackage[nameinlink, noabbrev, capitalise]{cleveref}
\usepackage{geometry,enumitem}
\usepackage{algorithm}
\usepackage{algpseudocode}

\usepackage{dsfont}
\usepackage{textpos}
\usepackage{thm-restate}
\usepackage{titling}
\usepackage{soul}
\usepackage{xcolor}

\usepackage{macros}
\usepackage{tikz}
\usetikzlibrary{shapes.geometric,hobby,calc}
\usetikzlibrary{shapes,decorations}
\usetikzlibrary{decorations.markings,intersections}
\usepackage{tikzmacros}

\makeatletter
\@addtoreset{claim}{theorem}
\makeatother

\hypersetup{
	colorlinks=true,
	linkcolor=myBlue,
	citecolor=myBlue,
	urlcolor=myLightBlue,
	bookmarksopen=true,
	bookmarksnumbered,
	bookmarksopenlevel=2,
	bookmarksdepth=3
}

\newif\ifcomment
\commentfalse
\commenttrue

\title{Torsoids in Path-Like Graphs without nontrivial even $2$-separation}
\author{Nathan Bowler
	\and 
	Florian  Reich
	\and
	Qiuzhenyu Tao
}

\preauthor{}
\DeclareRobustCommand{\authorthing}{
	\begin{center}
		Nathan Bowler -- Universität Hamburg, Germany\\
		\href{mailto:nathan.bowler@uni-hamburg.de}{nathan.bowler@uni-hamburg.de}
		
		\medskip
		Florian Reich -- Universität Hamburg, Germany\\
		\href{mailto:florian.reich@uni-hamburg.de}{florian.reich@uni-hamburg.de}
		
		\medskip
		Qiuzhenyu Tao -- Universität Hamburg, Germany\\
		\href{mailto:qiuzhenyu.tao@uni-hamburg.de}{qiuzhenyu.tao@uni-hamburg.de}
\end{center}}
\author{\authorthing}
\postauthor{}    

\begin{document}
	\maketitle

\begin{abstract}
    Bowler et al. introduced the concept of torsoids that describes the 1-separations of a directed graph in a canonical way building on previous progress by Lov{\'a}sz.
    To fully understand the structure of directed graphs with respect to their 1-separations it remains to find a global structure along which the torsoids are arranged.

    In this paper, we start the investigation of this global structure for a specific class of directed graphs.

\end{abstract}
	
	\section{Introduction}
	An unordered pair $\{A,B\}$ is a \textit{$k$-separation} of an undirected graph $G$ if $A \cup B = V(G)$, $|A\cap B|=k$, and there exists no edge between $A \setminus B$ and $B \setminus A$. 
	Graph separations are a fundamental concept in graph theory that has captivated mathematicians for decades.
	For small values of $k$, canonical combinatorial tree structures exist that describe the $k$-separations of a $k$-connected undirected graph and their relationships.
	Specifically, for $k = 0$, graphs trivially decompose into their connected components; for $k = 1$, the block-cut decomposition applies \cite[Chapter 3.1]{diestel17}; and for $k = 2$, this decomposition is known as the Tutte decomposition \cite{Tutte66}. Recent work by Carmesin and Kurkofka has extended this line of research to $k = 3$ \cite{Carmesin23}.

    In the context of directed graphs however, edges between $A \setminus B$ and $B \setminus A$ are permitted in one direction.
    Therefore, the structure of separations in directed graphs is more complex than in undirected graphs.
    For $k=0$ the resulting decomposition~\cite[Chapter 1.5]{BG2018digraph} forms a directed acyclic graph.
    Already for $k = 1$, the structure of $1$-separations is not fully understood. For many years, our understanding was limited to a partial result by Lov{\'a}sz \cite{Lovasz1987}.
    Recently, Bower et al.~\cite{bowler2023decomposition} improved Lov{\'a}sz result by representing the $1$-separations of a directed graph in a canonical way, called \emph{torsoids}.
    However, to fully understand the structure of $1$-separations, it remains to find a global structure along which the torsoids are arranged.
    In this paper, we start the investigation of such a global structure.
    
    As Lov{\'a}sz and Bowler et al., we work in the setting of \emph{matching covered graphs}, i.e.\ connected undirected graphs in which every edge is covered by some perfect matching.
    This is motivated by the well-established correspondence between strongly connected directed graphs and bipartite matching-covered graphs \cite{mccuaig2000evendicycles, rst1999pfaffianorientations}, which links directed $1$-separations to tight sets.
    
    Indeed, a torsoid in a matching covered graph $G$ is an equivalence class of tight set partitions of $G$.
    A \emph{tight set partition} of $G$ is a partition of $V(G)$ in which every partition class is a tight set.
    Furthermore, given a tight set partition $\cP$, let $\collapse{\cP}$ denote the graph obtained by contracting each tight set in $\cP$ and deleting any parallel edges resulting from this process.
    Note that, given a torsoid $\cT$, there is either a cycle $H$ of even length $k \geq 4$ or a noncyclic graph $H$ without nontrivial tight cuts such that $\collapse{\cP} = H$ for all tight set partitions $\cP$ corresponding to $\cT$.

    Bowler et al.~\cite{bowler2023decomposition} showed that torsoids are related to the torsos of a given maximal family of nested tight cuts $\cC$.    
    Two cuts in $\cC$ are said to be \emph{nested} if they do not cut through each other.
    A \emph{maximal star} of $\cC$ is a tight set partition $\cP$ of size at least $4$ whose cuts are $\cC$ such that $\collapse{\cP}$ has no nontrivial tight cuts, and we call $\collapse{\cP}$ a \emph{torso} of $\cC$.
    If $\collapse{\cP}$ is a $C_4$, then we call it a \emph{$C_4$-torso of $\cC$ at $\cP$}, or simply a \emph{$C_4$-torso}.
    Furthermore, $\collapse{\cP}$ is called a \emph{$C_4^i$-torso} if there are $i$ singleton tight sets in $\cP$ that are adjacent in $\collapse{\cP}$.



	
    Due to the complexity of finding a global structure along which the torsoids are arranged, we focus on matching covered graphs with a special maximal family of nested tight cuts $\cC$.
    More precisely, we consider graphs with a maximal family of nested tight cuts $\cC$ such that exactly two torsos of $\cC$ are $C_4^3$-torsos and all other torsos of $\cC$ are $C_4^2$-torsos. We refer to such graphs as \emph{path-like}.
    In this paper we find the desired global structure for path-like graphs without even $2$-separation (i.e., a $2$-separation where each proper side is nonempty and contains an even number of vertices).
    In the second paper of this series, we will generalize our result to arbitrary path-like graphs.

    The structure of this paper is as follows. We begin by introducing basic concepts in \Cref{sec:pre}. In \Cref{sec:path-like}, we define path-like graphs and establish their fundamental structural properties. In \Cref{sec:sourcesink}, we present a source-sink model for path-like graphs without nontrivial even $2$-separation that contain a distinguished edge, and we show how this model directly captures all torsoids and their interactions. In \Cref{sec:torsoidsG}, we extend this approach to path-like graphs without nontrivial even $2$-separation that do not contain a distinguished edge.

	\section{Preliminaries}\label{sec:pre}
    We follow~\cite{diestel17,BG2018digraph} for basic concepts and notations.
    A directed path from $u$ to $v$ is called a \textit{directed $u$--$v$ path}. We use $\operatorname{dist}(uv)$ to denote the number of edges in a shortest directed $u$--$v$ path.
    For two disjoint sets $X$ and $Y$ in a graph, $E(X,Y)$ denotes the set of edges with one endpoint in $X$ and the other in $Y$. 

    \begin{definition}[Directed Separation]\label{di-separation}
		Let $D$ be a directed graph. A tuple $(A, B)$ with $A, B \subseteq V (D)$ and $A \cup B = V(D)$ is a \emph{directed separation} of $D$ if there is no edge with tail in $A \setminus B$ and head in $B \setminus A$. The integer $k \coloneqq |A\cap B|$ is called the \emph{order} of the separation and we also refer to $(A, B)$ as a \emph{directed $k$-separation}.
	\end{definition}
    
    We refer to $A$ and $B$ as the \emph{sides} of the (directed) separation. The \emph{separator} of $(A,B)$ is the vertex set $A\cap B$.
    Moreover, for the directed $k$-separation $(A,B)$, we call $A \cap B$ the \emph{out-$k$-separator} of $A$ and the \emph{in-$k$-separator} of $B$.
    
    A separation $(A,B)$ is said to be \textit{trivial} if either $A \subseteq B$ or $B \subseteq A$, and \textit{decisive} if $|A \setminus B|, |B \setminus A| \geq 2$.
    Furthermore, we call a $(A,B)$ \emph{even} if $A\setminus B$ and $B \setminus A$ contain an even number of elements.
    Note that every nontrivial even separation is also decisive.

    
	Let $G$ be a bipartite graph with bipartition $\{V_1,V_0\}$ and let $M$ be a perfect matching of $G$.
    A path in $G$ is called an \emph{$M$-alternating path} if its edges are alternately in $M$ and not in $M$.
    We can construct a directed graph by directing all non-matching edges from $V_1$ to $V_0$ and contracting the matching edges to single vertices. The resulting graph is called the \textit{$M$-direction} of $G$ and is denoted $D(G, M)$.
    
	\begin{definition}[Matching Covered Graph]
		A graph $G$ is \emph{matching covered} if it is connected, has at least one edge, and for every edge $e\in E(G)$, there exists a perfect matching of $G$ containing $e$.
	\end{definition}

    
    Given a graph $G$ and a set $X\subseteq V (G)$, we call the set $\partial(X)\coloneqq \{e\in E(G)\colon |e\cap X|=1 \}$ the \emph{cut induced by $X$}.
    Let $G$ be a matching covered graph. A subset $X\subseteq V (G)$ is a \emph{tight set} in $G$ if every perfect matching of $G$ has precisely one edge in $\partial(X)$ and we refer to $\partial(X)$ as a \emph{tight cut}. Two tight cuts $C_1$ and $C_2$ are \emph{nested} if there exist disjoint sets $X_1, X_2 \subseteq V(D)$ such that $\boundary{X_1}=C_1$ and $\boundary{X_2}=C_2$.
    

  	\begin{proposition}[\cite{bowler2023decomposition}]\label{tightsettoseparation}
		The tight cuts in a matching covered bipartite graph $G$ with respect to a perfect matching $M$ correspond one-to-one to the $1$-separations in $D(G, M)$.
	\end{proposition}
    
    \begin{lemma}[\cite{bowler2023decomposition}]\label{crossing_tight_sets}
		Let \(X\) and \(X'\) be tight sets such that \(X \cap X'\) is odd. Then both \(X \cap X'\) and \(X \cup X'\) are tight sets and there is no edge with endpoints in both \(X \setminus X'\) and \(X' \setminus X\). If \(X \setminus X' \neq \emptyset\), then there is an edge with endpoints in \(X \cap X'\) and \(X \setminus X'\).
	\end{lemma}
    
    A graph without nontrivial tight cuts is called a \textit{brace} if it is bipartite, or a \textit{brick} otherwise. Such a graph is referred to as a \textit{BoB} (short for Brick or Brace).
    A set $S$ is \emph{passable} for $P$ if $P\cup S$ and $P\setminus S$ are tight.
	A partition $\mathcal{P}$ of the vertex set of a matching covered graph $G$ is a \emph{tight set partition} of $G$ if every $P\in \mathcal{P}$ is a tight set. For every tight set partition $\mathcal{P}$ of a matching covered graph $G$, a $\operatorname{coll}(\mathcal{P})$ is the graph with vertex set $\mathcal{P}$ and an edge between $P$ and $Q$ if and only if there are $p\in P$ and $q\in Q$ such that $pq\in E(G)$.
	
	
	\begin{definition}[Torso]\label{def_torso}
		Let $\mathcal{C}$ be a maximal family of nested tight cuts in $G$.
		A \emph{maximal star} of $\mathcal{C}$ is a tight set partition $\Partition$ of size at least $4$ such that $\collapse{\Partition}$ is a \BoB and $\boundary{P} \in \mathcal{C}$ for every $P \in \Partition.$
		Then we call $\collapse{\Partition}$ a \emph{torso of $\mathcal{C}$ at the maximal star $\mathcal{\Partition}$}, or a \emph{torso} for short.
		
		Furthermore, we simply call a torso at some maximal star of some maximal family of nested tight cuts in $G$ a \emph{torso in $G$}.
	\end{definition}
	If a torso is a $C_4$, we call it a \emph{$C_4$-torso}. Otherwise, the torso is a \BoB other than $C_4$, and we call it a \emph{non-$C_4$-torso}.
	Given a torso $\mathcal{S}$, we call a vertex $P$ of $\mathcal{S}$ \emph{single} if the tight set $P$ is a singleton. Further, let $i$ be the maximum number of adjacent single vertices in a $C_4$-torso $\mathcal{S}$, then we call $\cS$ a \emph{$C_4^i$-torso}.
	For any torso $\cS$, we call the union of single vertices of $\cS$ the \emph{flag} of $\cS$, denoted by $\mathcal{F}$.
    We remark that the flag of $\cS$ is a set of vertices of $G$.
	Note that the flags of different torsos of the same maximal family of nested tight cuts are disjoint.
	
	\begin{definition}[Torsoid]\label{def-torsoid}
		A \emph{torsoid} $\mathcal{T}$ in a matching covered graph $G$ is a pair $(H, \varepsilon)$ with:
		\begin{enumerate}
			\item[(T1)] $H$ is a matching covered graph on at least 4 vertices that is a BoB or a cycle
			\item[(T2)] the elements of $V(H)$ are tight sets of $G$
			\item[(T3)] $\varepsilon : E(H) \to 2^{V(G)}$
			\item[(T4)] $V(H) \cup \text{Im}(\varepsilon)$ is a near partition of $V(G)$, where $\text{Im}(\varepsilon)$ is the image of $\varepsilon$
			\item[(T5)] for $vw \in E(H)$, there is an edge from $v \cup \varepsilon(vw)$ to $w$ and $\varepsilon(vw)$ is largest among all subsets of $v \cup \varepsilon(vw) \cup w$ that are passable for both $v$ and $w$
			\item[(T6)] for $vw \notin E(H)$, there is no edge from $v$ to $w$ in $G$
			\item[(T7)] if $H$ is a cycle and $v$ is a vertex of $H$ with neighbours $u$ and $w$ then there is no partition of $\varepsilon(uv) \cup v \cup \varepsilon(vw)$ into tight sets $P_1, P_2, P_3$ such that both $u \cup P_1 \cup P_2$ and $P_2 \cup P_3 \cup w$ are tight.
		\end{enumerate}
		We call $\mathcal{T}$ \emph{cyclic} if $H$ is a cycle and \emph{noncyclic} otherwise.
	\end{definition}

    We say a torso $\mathcal{S}$ \emph{cleaves} a torsoid $\mathcal{T}$ in $G$ if every vertex of $\mathcal{S}$ contains a vertex of $\mathcal{T}$.
	Then by~\cite[Section 7]{bowler2023decomposition} we have following relations between torsos and torsoids that we will need:
	\begin{theorem} \label{rem:cyclic_torsoids}
		Let $G$ be a matching covered graph and let $\mathcal{C}$ be a maximal family of nested tight cuts in $G$. Then:
        \begin{enumerate}[label=(\arabic*)]
            \item\label{cleave-one} Each torso of $\mathcal{C}$ cleaves precisely one torsoid of $G$.
            \item\label{cleave-number} A cyclic torsoid with $|H|=n$ is cleaved by $\frac{n}{2}-1$ $C_4$-torsos of $\cC$.
            \item $G$ contains only cyclic torsoids if and only if all torsos of $\mathcal{C}$ are $C_4$-torsos.
        \end{enumerate}
	\end{theorem}
	
    
	
	
	
	\section{Path-like graphs}\label{sec:path-like}
	Let $G$ be a matching covered graph and $\mathcal{C}$ be a maximal family of nested tight cuts in $G$.
    Given a tree $T$, we let $\mathcal{L}(T)$ be the set of leaves of $T$ and set $\mathcal{I}(T):= V(T) \setminus \mathcal{L}(T)$.
	Further, let $\mathcal{P}(v)$ be the partition of $\mathcal{L}(T)$ into the vertices that are contained in a common component of $T-v$.
	
	\begin{definition}[Representation tree]\label{representation_tree}
		A \emph{representation tree} of $\mathcal{C}$ is a triple $(T, \alpha, \beta)_{\mathcal{C}}$, where $T$ is a tree, $\alpha$ is a bijection between $\mathcal{L}(T)$ and $V(G)$ and $\beta$ is a bijection between $\mathcal{I}(T)$ and the maximal stars of $\mathcal{C}$ such that $\alpha(\mathcal{P}(v)) = \beta(v)$ for any $v \in \mathcal{I}(T)$.
	\end{definition}
	
	\begin{proposition}
		There exists a unique representation tree of $\mathcal{C}$.
	\end{proposition}
	\begin{proof}
        Let $\mathbb{S}$ be the family of all torsos based on $\cC$.
		We construct the desired representation tree by recursion on $|\mathbb{S}|$.
		For $|\mathbb{S}|=1$, let $T$ be a star with $|V(G)|$ with $4$ leaves.
		Further, let $\alpha$ map the leaves to distinct vertices of $G$ and let $\beta$ map the non-leave vertex of $T$ to the unique maximal star of $G$.
		Note that $(T, \alpha, \beta)$ is the desired represenation tree of $G$.
		For $|\mathbb{S}| > 1$, pick some non-singleton set $Y \subseteq V(G)$ with $\boundary{Y} \in \mathcal{C}$. Let $G_1$ be the graph obtained by contracting $Y$ to a single vertex $y_1$ and let $G_2$ be the graph obtained by contracting $V(G) \setminus Y$ to a single vertex $y_2$.
		Further, let $(T_1, \alpha_1, \beta_1)$ and $(T_2, \alpha_2, \beta_2)$ be the representation trees of $G_1$ and $G_2$ with respect to the corresponding subfamilies of $\mathcal{C}$.
		Finally, let $T$ be the tree obtained from $T_1-y_2$ and $T_2-y_1$ by adding an edge incident with the maximal star of $G_1$ that is incident with $y_2$ in $T_1$ and the maximal star of $G_2$ that is incident with $y_1$ in $T_2$.
		Then $(T, \alpha, \beta)$ is the desired representation tree, where $\alpha$ and $\beta$ extend the maps $\alpha_1, \alpha_2, \beta_1, \beta_2$, respectively.
		Every representation tree of $\mathcal{C}$ has to obey the conditions used for this recursion.
		Thus the representation tree is indeed unique.
	\end{proof}

	\begin{definition}\label{path-like}
		Given a representation tree $(T, \alpha, \beta)_\mathcal{C}$ of $\mathcal{C}$, we say $\mathcal{C}$ is \emph{path-like} if the subgraph $T[\mathcal{I}(T)]$ is a path with only $C_4^3$- and $C_4^2$-torsos.
        If such a path-like $\mathcal{C}$ exists in $G$, then we say $G$ is \emph{path-like}.
	\end{definition}
	
	From \Cref{def_torso,representation_tree,path-like,rem:cyclic_torsoids}, we can get some properties of path-like graphs as following.
	\begin{observation}\label{ob_path-like}
		Let $G$ be a path-like graph, $\mathcal{C}$ be a path-like maximal family of nested tight cuts in $G$ and $(T, \alpha, \beta)_{\mathcal{C}}$ be the representation tree of $\mathcal{C}$. Then the following properties hold:
		\begin{enumerate}[label=(\arabic*)]
			\item $G$ is a matching covered bipartite graph with only cyclic torsoids.
			\item\label{obs:order_of_path} The maximal stars in $\mathcal{C}$ admits a canonical linear ordering induced by the path $T[\mathcal{I}(T)]$, denoted $\cS_1,\cdots,\cS_n$.
			\item \label{obs:adjacent} Each maximal star $\cS_i$ contains at least two adjacent single vertices.
			\item \label{ob_path1} $\beta$ maps the endpoints of the path $T[\mathcal{I}]$ to $C_{4}^3$-torsos and all other vertices of $T[\mathcal{I}]$ to $C_4^2$-torsos.
			\item \label{obs:segment_of_path} For every initial or terminal segment $J$ of the path $T[\mathcal{I}]$, the subgraph of $G$ induced by all flags of $J$ is connected.
		\end{enumerate}
	\end{observation}
	In particular, there are precisely two $C_4^3$-torsos of $\mathcal{C}$, and all the other torsos of $\mathcal{C}$ are $C_4^2$-torsos.
	We denote the flag of $\mathcal{S}_i$ as $\mathcal{F}_i$. Note that $\bigcup_{i \in [n]} \mathcal{F}_i$ is a partition of $V(G)$, and the flag $\mathcal{F}_i$ has size 3 if $i\in\{1,n\}$ and size 2 otherwise.

    \begin{observation} \label{obs:deletion_of_flag_edge}
        For every edge $xy \in E(G)$ with $x,y \in \cF_i$ for some $i \in [n]$, the graph $G - x - y$ is connected.
    \end{observation}
    
    By the property of torsos, we can get the following two observations about flags.

	\begin{observation}\label{ob_flag}
		For any vertex $x$ in $\mathcal{F}_i$, $N(x)\cap (V(G)\setminus \mathcal{F}_i)$ is non-empty, and is contained either entirely in the preceding flags or entirely in the succeeding flags.
	\end{observation}
	
	In the former case, we call $x$ the \textit{upper vertex} of $\mathcal{F}_i$ and in the latter case the \textit{lower vertex} of $\mathcal{F}_i$.
	Furthermore, we call the unique vertex of $\mathcal{F}_1$ ($\mathcal{F}_n$) that is only adjacent to elements of $\mathcal{F}_1$ ($\mathcal{F}_n$) an \emph{upper vertex} (\emph{lower vertex}).

    \begin{observation}\label{ob_flag_2}
        For each $i \in [n-1] \setminus \{1\}$, the two vertices in $\cF_i$ are classified as the upper vertex and the lower vertex, respectively.
    \end{observation}

	
	\subsection{$2$-separations in path-like graphs}
    Let $G$ be a path-like graph and let $\mathcal{C}$ be a path-like maximal family of nested tight cuts in $G$.
    
	\begin{lemma}\label{lem:decisive}
		If $G$ has a $C_k$-torsoid with $k\ge6$, then there exists a decisive $2$-separation of $G$.
	\end{lemma}
	\begin{proof}
		Let $\cT=(H, \edgefkt{})$ be a $C_k$-torsoid in $G$ with $k\ge6$. By \Cref{rem:cyclic_torsoids}-\Cref{cleave-number}, there exist two distinct torsos $\cS_1$ and $\cS_2$ both cleaving $\cT$ in $G$. Then \Cref{ob_path-like}-\Cref{obs:adjacent} ensures that each torso contains a pair of adjacent single vertices, say $p_1q_1\in \cS_1$ and $p_2q_2\in \cS_2$, which implies $p_1q_1,p_2q_2\in H$. Since $H$ is $C_k$ with $k\ge6$, every other vertex of $H$ lies between a vertex in $p_1q_1$ and a vertex in $p_2q_2$. Thus, deleting a suitable pair of two non-adjacent vertices from $\{p_1,q_1,p_2,q_2\}$ disconnects the graph $G$, leaving each side with at least two vertices.
        This implies that we can deduce a decisive $2$-separation of $G$ with the separator in $\{p_1,q_1,p_2,q_2\}$.
	\end{proof}
	
	\begin{proposition}\label{prop:decisive_distinct_flags}
		The vertices of the separator of every decisive $2$-separation are contained in distinct flags of $\mathcal{C}$.
	\end{proposition}
	\begin{proof}
		Let $(A, B)$ be a decisive $2$-separation in $G$ with $\{x, y\}:= A \cap B$. Note that \Cref{ob_path-like}-\Cref{ob_path1} ensures that $\cS_1$ and $\cS_n$ are $C_4^3$-torsos and any other $\cS_i$ is $C_4^2$-torso.
		Suppose for a contradiction that $x,y$ are contained in the same flag, say $\mathcal{F}_i$.
        If $xy\notin E(G)$, then $i\in \{1,n\}$, which means that one proper side of $(A,B)$ has size one, contradicting that $(A, B)$ is decisive.
        Otherwise, $\{x\},\{y\}$ are adjacent vertices of the corresponding torso $\cS_i$, which implies that $G - x - y$ is connected \cref{obs:deletion_of_flag_edge}.
		Then either $A$ or $B$ is the empty set, contradicting that $(A, B)$ is decisive.
	\end{proof}
	
	\begin{lemma}\label{lem-decisive-even}
		Every decisive $2$-separation in $G$ can be refined to a nontrivial even $2$-separation.
	\end{lemma}
	\begin{proof}
		Let $(A, B)$ be an odd decisive $2$-separation in $G$ with $\{x, y\}:= A \cap B$.
		By~\Cref{prop:decisive_distinct_flags}, we can assume that $x, y$ are contained in distinct flags, say $x \in \mathcal{F}_i$ and $y \in \mathcal{F}_j$ for $i < j$.
        \begin{claim} \label{clm:upper_or_lower}
			Either $x$ is the upper vertex of $\mathcal{F}_i$ or $y$ is the lower vertex of $\mathcal{F}_j$.
		\end{claim}
		\begin{claimproof}
			Suppose for a contradiction that $x$ is a lower vertex of $\mathcal{F}_i$ and $y$ is an upper vertex of $\mathcal{F}_j$.
            Then $\bigcup_{k \leq i} \mathcal{F}_k \setminus \{x\}$ induces a connected subgraph of $G$ by~\Cref{ob_path-like}-\Cref{obs:segment_of_path}, \Cref{ob_flag,ob_flag_2}, and is even by \Cref{ob_path-like}-\Cref{obs:order_of_path,ob_path1}.
            By a similar argument, $\bigcup_{k \geq j} \mathcal{F}_k \setminus \{y\}$ is an even set that induces a connected subgraph of $G$.
			Since any flag $\mathcal{F}_k$ with $i < k < j$ has even size and induces a connected subgraph by \Cref{ob_path-like}-\Cref{obs:segment_of_path,ob_path1}, the separation $(A, B)$ is even, a contradiction.
		\end{claimproof}
		By~\Cref{clm:upper_or_lower}, we can assume without loss of generality that $x$ is the upper vertex of $\mathcal{F}_i$.
		Since $G- y$ is connected, the vertex $x$ has to be incident with $A \setminus \{x,y\}$ and $B \setminus \{x,y\}$.
		If $i = 1$, then the two lower vertices of $\mathcal{F}_1$ belong respectively to $A$ and $B$, since $x$ is only adjacent to these two vertices. Let $z$ be the one in $A$. Then $(A \setminus \{x\}, B \cup \{z\})$ is a nontrivial even $2$-separation as desired.
		If $i > 1$, then let $z$ be the unique lower vertex of $\mathcal{F}_i$.
		By~\Cref{ob_path-like}-\Cref{obs:segment_of_path}, \Cref{ob_flag,ob_flag_2}, $\bigcup_{k < i} \mathcal{F}_k$ is either contained in $A$ or $B$. Since $x$ is only incident with $z$ and $\bigcup_{k < i} \mathcal{F}_k$, we can choose the desired $2$-separation as $(A \cup \{z\}, B \setminus \{x\})$ in the former case, and $(A \setminus \{x\}, B \cup \{z\})$ in the latter case. Here we are done.
	\end{proof}

    \section{Graphs with distinguished edge}\label{sec:sourcesink}


    
	In this section, we start to analyze the torsoids in path-like graphs without nontrivial even $2$-separation.
    Note that such graphs contain only $C_4$-torsoids by~\Cref{lem:decisive,lem-decisive-even}.
	We begin by considering the subclass of such graphs that have a \emph{distinguished edge}, as their torsoids have a slightly simpler structure:
	An edge is \emph{distinguished} if it has one endpoint in the flag $\cF_1$ and one endpoint in the flag $\cF_n$.
    Then in \Cref{sec:torsoidsG}, we turn our attention to arbitrary path-like graphs without nontrivial even $2$-separation.

    Let $G$ be a path-like graph with a distinguished edge and no nontrivial even $2$-separations, and let $\cC$ be a path-like maximal family of nested tight cuts in $G$.
    We remark that \Cref{ob_path-like,ob_flag,ob_flag_2} ensure that the graph $G$ is bipartite with partition classes $V_1$ and $V_2$, where $V_1$ consists of the lower vertices of the flags of $\cC$ and $V_2$ consists of the upper vertices of the flags of $\cC$.
    
	Let $e=st$ be a distinguished edge of $G$ with $s \in \cF_1$ and $t \in \cF_n$.
	By~\Cref{ob_path-like}, each $G[\cF_i]$ contains a unique edge $f_i$ for $i \in [n-1] \setminus \{1\}$ and also $G[\cF_1\setminus\{s\}]$ and $G[\cF_n\setminus\{t\}]$ contain unique edges $f_1$ and $f_n$, respectively. For each $i\in [n]$, denote the lower and upper vertex of $f_i$ in $G$ by $x_i^+\in V_1$ and $x_i^-\in V_2$, respectively.
	We define $M_{\cC}^e:=\{f_i: i \in [n]\} \cup \{e\}$.
	Note that $M_{\cC}^e$ is a perfect matching of $G$.

    \begin{proposition}\label{prop:unique_perfect_matching}
		$M_{\cC}^e$ is the unique perfect matching of $G$ containing $e$.
	\end{proposition}
	\begin{proof}
		Let $M$ be an arbitrary perfect matching in $G$ that contains $e$.
		By definition, the upper vertex of $\cF_1$ is only incident with lower vertices of $\cF_1$.
		This implies that $M$ contains $f_1$.
		We assume that $M$ contains $f_1, \dots, f_i$ for some $i \in [n]$.
		If $i=n$, then $M = M_{\cC}^e$, as desired.
		Otherwise, the upper vertex of $\cF_{i+1}$ (that is not $t$, if $i+1 = n$) is only incident with $f_{i+1}$ and edges with an endpoint in $\bigcup_{j \in [i]}\cF_j$ by \Cref{ob_flag,ob_flag_2,ob_path-like}.
		This shows $f_{i+1} \in M$ and we continue recursively.
	\end{proof}

    \begin{corollary}\label{cor:m_alternating_cycle}
		Every $M_\cC^e$-alternating cycle contains $e$.
	\end{corollary}
	\begin{proof}
		Suppose for a contradiction that $C$ is an $M_\cC^e$-alternating cycle that avoids $e$.
		Then the symmetric difference of $C$ and $M_{\cC}^e$ is a perfect matching in $G$ that contains $e$, contradicting~\Cref{prop:unique_perfect_matching}.
	\end{proof}

	Given an $M_\cC^e$-alternating $s$--$t$~path $P$ in $G-e$, let $\Delta(P):= (E(P) \cup \{e\}) \triangle M_\cC^e$.
	\begin{proposition}\label{prop:alternating_s_t_paths}
		The function $\Delta$ maps the $M_\cC^e$-alternating $s$--$t$~paths in $G-e$ one-to-one to the perfect matchings in $G-e$.
	\end{proposition}
	\begin{proof}
		Let $M$ be any perfect matching of $G-e$. 
        \cref{prop:unique_perfect_matching} ensures that $M$ is distinct to $M_\cC^e$. Note that the symmetric difference of two distinct perfect matchings is a vertex-disjoint union of alternating cycles. By \cref{cor:m_alternating_cycle}, $M$ is equal to the symmetric difference of $M_\cC^e$ and an $M_\cC^e$-alternating cycle containing $e$. This means there exits a one-to-one correspondence between perfect matchings of $G-e$ and $M_\cC^e$-alternating cycles containing $e$. 
        Then by removing the edge $e$, such cycles correspond one-to-one to $M_\cC^e$-alternating $s$--$t$~paths in $G-e$. This completes the proof.
	\end{proof}

    \subsection{Source-sink model}

    Let $D(G,M_\cC^e)$ be the $M_{\cC}^e$-direction of $G$ with the direction from $V_1$ to $V_2$.
	\begin{definition}[Source-sink model]\label{def-ssm}
        We construct a directed graph from $D(G,M_\cC^e)$ by splitting the vertex $v_e$ (which results from contracting $e$) into a source $s$ and a sink $t$. In this process, all out-edges of $v_e$ are assigned to $s$, and all in-edges to $t$. The resulting graph is called \emph{source-sink model} of $G$ and is denoted by $D_\cC^e$.
	\end{definition}
	The construction of the source-sink model of $G$ ensure that we can represent the torsoids of $G$ and their interactions in a simple way, which we will elaborate in this section. First, we examine some properties of the source-sink model.
    
    For each $i\in[n]$, let $x_i$ be the vertex in $D_\cC^e$ contracted by $f_i$. We note that all $x_i$ correspond to vertices of $D(G,M_\cC^e)$.
    Furthermore, we note that $s$ is a lower vertex of $\cF_1$ and $t$ is an upper vertex of $\cF_n$.

    \begin{proposition}\label{remark:alternating_path}
        Every $M_\cC^e$-alternating $s$--$t$~path in $G-e$ traverses every edge $f_i$ of $M_\cC^e$ from upper to lower vertex in the flag $\cF_i$, and every edge $f=uv$ of $E(G) \setminus M_\cC^e$ from $u$ as a lower vertex of some flag $\cF_j$ and $v$ as an upper vertex of a distinct flage $\cF_k$ with $j<k$.
    \end{proposition}
    \begin{proof}
        By \Cref{ob_flag,ob_flag_2,ob_path-like}.
    \end{proof}

    \begin{corollary}\label{alternating_to_dipath}
        There exists a one-to-one correspondence between $M_\cC^e$-alternating $s$--$t$~paths in $G-e$ and directed $s$--$t$~paths in $D_\cC^e$.
    \end{corollary}
    \begin{proof}
        Directly by \Cref{def-ssm,remark:alternating_path} and the definition of $D(G,M_\cC^e)$.
    \end{proof}
    
	\begin{proposition}\label{prop:acyclic_ordering}
		The source-sink model $D_\cC^e$ is an acyclic directed graph.
        Furthermore, $s,x_1,x_2,\cdots,x_n,t$ is an acyclic ordering of $D_\cC^e$.
	\end{proposition}
	\begin{proof}
		Suppose for a contradiction that $D_\cC^e$ contains a cycle $C$.
		Then $C$ corresponds to an $M_{\cC}^e$-alternating cycle $C'$ in $G$ that avoids $e$, contradicting \cref{cor:m_alternating_cycle}.

        We now turn to the furthermore part.
        \Cref{def-ssm} ensures that all edges incident with $s$ are out-going and all edges incident with $t$ are in-going.
		Let $x_ix_j$ be an arbitrary edge in $D_\cC^e - s - t$. Then $x_ix_j$ corresponds to an edge in $G$ between the lower vertex of $\mathcal{F}_i$ and the upper vertex of $\mathcal{F}_j$ by the definition of $D(G,M_\cC^e)$. This implies that $i < j$ by \Cref{ob_flag}, which completes the proof.
	\end{proof}
    Note that $s$ is the only source and $t$ is the only sink of $D_\cC^e$ by \Cref{ob_flag,ob_flag_2}.

    \begin{corollary}\label{cor:path_from_s_and_to_t}
        For every vertex $x \in V(D_\cC^e)$ there is a directed $s$--$x$~path and a directed $x$--$t$~path in $D_\cC^e$.
    \end{corollary}

    \begin{proposition}\label{prop:source_sink_no_2_separation}
        There exist two internally-disjoint directed $s$--$t$~path in $D_\cC^e$.
    \end{proposition}
    \begin{proof}
        Suppose not for a contradiction. Then Menger's theorem ensures the existence of a vertex $x_i \in V(D_\cC^e) \setminus \{s,t\}$, such that there is no directed $s$--$t$~path in $D_\cC^e - x_i$. By~\Cref{alternating_to_dipath}, every directed $s$--$t$~path in $D_\cC^e$ corresponds to an $M_\cC^e$-alternating $s$--$t$~path in $G-e$. Thus every $M_\cC^e$-alternating $s$--$t$~path avoiding $e$ contains $f_i$.
        Since $G$ is matching covered, there exists a perfect matching $N$ in $G-e$ that contains $f_i$.
        By~\cref{prop:alternating_s_t_paths}, there is an $M_\cC^e$-alternating $s$--$t$~path $P$ with $\triangle{P}=N$.
        Thus $P$ avoids $e$ and $f_i$ and thus, $P$ induces a directed $s$--$t$~path in $D_\cC^e - x_i$, a contradiction.
    \end{proof}

    

    \begin{corollary}\label{lem:two_disjoint_paths}
        There exist two internally disjoint $M_\cC^e$-alternating $s$--$t$~paths in $G$ that avoid~$e$.
    \end{corollary}

    \subsection{Torsoids}\label{sec:torsoidsGplus}
	In this subsection, we build a one-to-one correspondence between torsoids in $G$ and the following near partitions of $D_\cC^e$:
	\begin{definition}\label{def-partition-ssm}
        For every vertex $x$ in $D_\cC^e$ except the source $s$ and sink $t$, we define $\cP(x):=\{\{x\},I^x,I_p^x,\{i^x\},R^x,\{o^x\},O_p^x,O^x\}$ in $D_\cC^e$ as follows:
		\begin{itemize}
            \item $i^x$: The vertex of $V(D_\cC^e) \setminus \{x\}$ minimizing $\operatorname{dist}(i^x,x)$ such that there is no directed $s$--$x$~path in $D_\cC^e - i^x$.
            \item $o^x$: The vertex of $V(D_\cC^e) \setminus \{x\}$ minimizing $\operatorname{dist}(x,o^x)$ such that there is no directed $x$--$t$~path in $D_\cC^e - o^x$. 
			\item $I^x$: The set of vertices $v \in V(D_\cC^e) \setminus \{x\}$ such that there is no directed $v$--$t$~path in $D_\cC^e - x$.
			\item $O^x$: The set of vertices $v \in V(D_\cC^e) \setminus \{x\}$ such that there is no directed $s$--$v$~path in $D_\cC^e - x$.
            \item $I_p^x$: The set of all internal vertices of directed $i^x$--$x$~paths that are not in $I^x$.
            \item $O_p^x$: The set of all internal vertices of directed $x$--$o^x$~paths that are not in $O^x$.
			\item $R^x$: All remaining vertices of $D_\cC^e$.
		\end{itemize}
	\end{definition}

    \begin{proposition}\label{near-partition}
        The $\cP(x)$ is indeed a near-partition of $V(D_\cC^e)$.
    \end{proposition}
    \begin{proof}
        \Cref{def-partition-ssm} ensures that $x$ and $R^x$ are disjoint from all other classes of $\cP(x)$.
        By construction, every element of $I^x \cup I_p^x \cup \{i^x\}$ can reach $x$ and every element of $O^x \cup O_p^x \cup \{o^x\}$ can be reached from $x$.
        This implies that the sets $I^x \cup I_p^x \cup \{i^x\}$ and $O^x \cup O_p^x \cup \{o^x\}$ are disjoint, $s\notin O^x \cup O_p^x \cup \{o^x\}$ and $t\notin I^x \cup I_p^x \cup \{i^x\}$ by \Cref{prop:acyclic_ordering}. 
        
        Furthermore, the set $I_p^x$ is disjoint from $I^x$ and $\{i^x\}$ by choice.
        By~\cref{prop:source_sink_no_2_separation}, $s \notin I^x$, which implies that the vertex $i^x$ is not contained in $I^x$.
        Thus the sets $I^x$, $I_p^x$ and $\{i^x\}$ are pairwise disjoint.
        By symmetry, also $O^x$, $O_p^x$ and $\{o^x\}$ are pairwise disjoint, which completes the proof.
    \end{proof}

    By \Cref{prop:acyclic_ordering}, we can deduce:
    \begin{observation}\label{obs:PX_properties}
    The following properties hold for $\cP(x)$: Let $v \in V(D_\cC^e)$ be arbitrary.
    \begin{enumerate}[label=(\arabic*)]
        \item\label{obs:PX1} The vertex $v$ is contained in $I^x \cup I_p^x$ if and only if $v$ is an internal vertex of some directed $i^x$--$x$~path.
        \item\label{obs:PX2} The vertex $v$ is contained in $O^x \cup O_p^x$ if and only if $v$ is an internal vertex of some directed $x$--$o^x$~path.
    \end{enumerate}
    \end{observation}
    
	\begin{proposition}\label{prop:inoutsep}
		For a near partition $\cP(x)$, the following pairs are all directed $1$-separations of $D_\cC^e$:
        \begin{enumerate}
            \item[(P1)] $(\{x\} \cup I^x, V(D_\cC^e) \setminus I^x)$, $(V(D_\cC^e) \setminus I_p^x, \{i^x\} \cup I_p^x)$, $(V(D_\cC^e) \setminus (I^x \cup I_p^x),I^x \cup I_p^x \cup \{i^x\})$ and $(V(D_\cC^e) \setminus (\{x\} \cup I^x \cup I_p^x),\{x\} \cup I^x \cup I_p^x \cup \{i^x\})$;
            \item[(P2)] $(V(D_\cC^e) \setminus O^x, \{x\} \cup O^x)$, $(\{o^x\} \cup O_p^x, V(D_\cC^e) \setminus O_p^x)$, $(O^x \cup O_p^x \cup \{o^x\}, V(D_\cC^e) \setminus (O^x \cup O_p^x))$ and $(\{x\} \cup O^x \cup O_p^x \cup \{o^x\}, V(D_\cC^e) \setminus (\{x\} \cup O^x \cup O_p^x))$.
        \end{enumerate}
	\end{proposition}
	\begin{proof}
		Up to symmetry, it suffices to prove any pair in (P1) is a directed $1$-separation of $D_\cC^e$.
        
        \textbf{Case 1: $(\{x\} \cup I^x, V(D_\cC^e) \setminus I^x)$.}
        First, suppose for a contradiction that there exists an edge $f$ with tail in $I^x$ and head outside $I^x \cup \{x\}$.
        Let $y$ be the head of $f$.
        By the definition of $I^x$ in \Cref{def-partition-ssm}, there is no directed $y$--$t$~path in $D_\cC^e - x$, which means that $y\in I^x \cup \{x\}$, a contradiction. Thus $(\{x\} \cup I^x, V(D_\cC^e) \setminus I^x)$ is a directed $1$-separation of $D_\cC^e$.

        \textbf{Case 2: $(V(D_\cC^e) \setminus I_p^x, \{i^x\} \cup I_p^x)$.}
        Suppose that there exists an edge $f'$ with tail outside $\{i^x\} \cup I_p^x$ and head in $I_p^x$.
        Let $z$ be the tail of $f'$.
        Then there exists a directed $s$--$z$~path $P_1$ in $D_\cC^e$ since $s$ is the only source by \Cref{prop:acyclic_ordering}.
        By the definition of $I_p^x$ and the choice of $f'$, there exists a directed $z$--$x$~path $P_2$ in $D_\cC^e$ that avoids $i^x$.
        Note that $P_1$ and $P_2$ concatenate to a directed $s$--$x$~path by \Cref{prop:acyclic_ordering}, then $i^x$ is either the startvertex or an internal vertex of $P_1$ by the definition of $i^x$. This implies that $z$ is contained in $I_p^x\cup I^x$ by \Cref{obs:PX_properties}-\Cref{obs:PX1}. Since $z$ is the tail of $f'$, $z\in I^x$. Then there exists a directed $z$--$t$ path in $D_\cC^e$ that avoids $x$ since the head of $f'$ is in $I_p^x$, a contradiction to the definition of $I^x$. 
        This shows that $(V(D_\cC^e) \setminus I_p^x, \{i^x\} \cup I_p^x)$ is a directed $1$-separation of $D_\cC^e$.

        \textbf{Case 3: $(V(D_\cC^e) \setminus (I^x \cup I_p^x),I^x \cup I_p^x \cup \{i^x\})$ and $(V(D_\cC^e) \setminus (\{x\} \cup I^x \cup I_p^x),\{x\} \cup I^x \cup I_p^x \cup \{i^x\})$.}
        Let $A=I^x \cup I_p^x \cup \{i^x\}$ or $\{x\} \cup I^x \cup I_p^x \cup \{i^x\})$.
        Suppose that there exists an edge $f''$ with tail outside $A$ and head in $A\setminus \{i^x\}$. Let $z'$ be the tail of $f''$. 
        Similar to \textbf{Case 2}, $z'$ is contained in $I_p^x\cup I^x$, a contradiction. 
        This completes the proof.
	\end{proof}

    
    We call a directed $1$-separation $(A,B)$ of $D_\cC^e$ \emph{good} if $s$ and $t$ are in the same side, i.e.\ $s,t \in A$ or $s,t \in B$, while $A\neq\{s\}$ and $B\neq\{t\}$. Note that all directed 1-separations in \Cref{prop:inoutsep} are good.
    By~\cref{cor:path_from_s_and_to_t} we know:
    \begin{observation} \label{obs:s_and_t}
        Let $(A,B)$ be a good directed $1$-separation of $D_\cC^e$.
        If $s\in A \cap B$, then $t\in A \setminus B$ and if $t\in A \cap B$, then $s\in B \setminus A$.
    \end{observation}
    Given a good directed $1$-separation $(A,B)$ of $D_\cC^e$,
    we define:
    \begin{equation*}
        \phi(A,B):=
        \begin{cases}
            ((A\setminus\{s,t\})\cup{\{v_e\}}, B), & \text{ if }s,t\in A \setminus B;\\
            (A, (B\setminus\{s,t\})\cup{\{v_e\}}), &\text{ if } s,t\in B \setminus A;\\
            ((A\setminus\{s,t\})\cup \{v_e\}, (B\setminus \{s\})\cup \{v_e\}), &\text{ if } s \in A \cap B, t\in A \setminus B;\\
            ((A\setminus \{t\})\cup \{v_e\}, (B\setminus \{s,t\})\cup \{v_e\}), &\text{ if } t\in A \cap B, s \in B \setminus A,
        \end{cases}
    \end{equation*}
    where $v_e$ is the vertex corresponding to $e$ in $D(G,M_\cC^e)$.


    \begin{proposition}\label{prop:surmap}
        The function $\phi$ is a surjective map from the set of good directed $1$-separations of $D_\cC^e$ to the set of directed $1$-separations of $D(G,M_\cC^e)$.
    \end{proposition}
    \begin{proof}
        Let $(A,B)$ be a good directed $1$-separation of $D_\cC^e$.
        Suppose for a contradiction that there exists an edge $f$ in $D(G,M_\cC^e)$ with tail in the left side of $\phi(A,B)$ and head in the right side of $\phi(A,B)$ that avoids the separator of $\phi(A,B)$.
        Since $(A,B)$ is a directed $1$-separation, one endpoint of $f$ is $v_e$, which implies $v_e$ is not the separator of $\phi(A,B)$. 
        Then $s,t \in A \setminus B$ or $s,t \in B \setminus A$. Therefore, $f$ corresponds to an edge in $D_\cC^e$ with tail in $A \setminus B$ and head in $B \setminus A$, contradicting that $(A,B)$ is a directed $1$-separation.
        
        For the other direction, let $(A,B)$ be a directed $1$-separation of $D(G,M_\cC^e)$.
        If $v_e\in A \setminus B$, then $((A\setminus \{v_e\})\cup \{s,t\},B)$ is a good directed $1$-separation of $D_\cC^e$ satisfying $\phi((A\setminus \{v_e\})\cup \{s,t\},B)=(A,B)$. 
		If $v_e \in B \setminus A$, then $(A, (B \setminus \{v_e\}) \cup \{s,t\})$ is a good directed $1$-separation of $D_\cC^e$ satisfying $\phi(A, (B \setminus \{v_e\}) \cup \{s,t\})= (A,B)$.
        Otherwise, $v_e \in A \cap B$, then both $((A\setminus \{v_e\})\cup \{s,t\}, (B\setminus \{v_e\})\cup \{s\})$ and $((A\setminus \{v_e\})\cup \{t\}, (B\setminus \{v_e\})\cup \{s,t\})$ are distinct good directed $1$-separations of $D_\cC^e$, and both are preimages of $(A,B)$ under $\phi$.
        This completes the proof.
    \end{proof}
    Note that a directed 1-separation $(A,B)$ of $D(G,M_\cC^e)$ admits two preimages under $\phi$ in $D_\cC^e$ precisely when $v_e$ is the separator.
    
    Let $(A_1,A_2)$ be a good directed 1-separation of $D_\cC^e$. According to \Cref{tightsettoseparation,prop:surmap},
    the separation $(A_1,A_2)$ corresponds to a unique tight cut in $G$, which we denote by $\rho(A_1,A_2)$.
    Moreover, each side $A_i$ corresponds to a tight set in $G$, which we denote by $\sigma(A_i)$, such that $\rho(A_1,A_2) = \boundary{\sigma(A_1)}=\boundary{\sigma(A_2)}$.

    \begin{proposition}\label{prop:cut-side2}
		Let $(A_1,A_2)$ be a good directed $1$-separation of $D_\cC^e$. Every edge in $\rho(A_1,A_2)$ has one endpoint in $\sigma(A_1)$ as an upper vertex of some flag, and the other in $\sigma(A_2)$ as a lower vertex of some flag.
	\end{proposition}
	\begin{proof}
		Let $x \in V(D(G,M_\cC^e))$ be the separator of $\phi(A_1,A_2)$ and let $e_x \in M_\cC^e$ be the corresponding edge in $G$. Then $e_x\in \rho(A_1,A_2)$.
        Let $f$ be an arbitrary edge in $\rho(A_1,A_2)\setminus \{e_x\}$. Then $f$ corresponds to an edge $f'\in E(A_1\setminus\{x\},A_2\setminus\{x\})$.
        By \Cref{di-separation}, $f'$ is an edge from $A_2\setminus\{x\}$ to $A_1\setminus\{x\}$, which implies that the endpoint of $f$ in $\sigma(A_1)$ is an upper vertex in some flag, and the other one in $\sigma(A_2)$ is a lower vertex of a distinct flag.

        Now we consider the case for the edge $e_x$.
        If $x=v_e$, then $e_x=e=st$. By \Cref{obs:s_and_t}, we have that $s\in \sigma(A_2)$ and $t\in \sigma(A_1)$. Note that $s$ is a lower vertex of $\cF_1$ and $t$ is an upper vertex of $\cF_n$.
        Otherwise, $x=x_i$ for some $i\in[n]$, then $e_x=f_i=x_i^+x_i^-$ with $x_i^+$ the lower vertex of $\cF_i$ and $x_i^-$ the upper vertex of $\cF_i$.
        Note that $x$ is also the separator of $(A_1,A_2)$, and both $s$ and $t$ are either in $A_1\setminus A_2$ or $A_2\setminus A_1$ by the definition of $\phi(A_1,A_2)$.
        Since $s$ is the only source and $t$ is the only sink of $D_\cC^e$, for any vertex $v$ that is not on the same side as $s$ and $t$, there exists a directed $s$--$v$ path $P_s$ and a directed $v$--$t$ path $P_t$ in $D_\cC^e$. By \Cref{prop:acyclic_ordering}, $P_s$ and $P_t$ are disjoint, which means one of them goes through $x$ and the other contributes an edge from $A_2\setminus A_1$ to $A_1\setminus A_2$.
        
        Next, we will show that $x_i^-\in \sigma(A_1)$ and $x_i^+\in \sigma(A_2)$. Suppose not for a contradiction. 
        If $s,t\in A_1\setminus A_2$, then $x$ is in $P_s$. By \Cref{alternating_to_dipath}, $P_s$ corresponds to a $M_\cC^e$-alternating $s$--$v$ path $P_s'$ in $G-e$. Moreover, $P_s'$ goes from $s$ to $x_i^-$ first, then to $x_i^+$, and goes from $x_i^+$ to $v$ by \Cref{remark:alternating_path}.
        Since $x_i^+\in \sigma(A_1)$ and $x_i^-\in \sigma(A_2)$, $|\Delta(P_s')\cap \rho(A_1,A_2)|>1$, a contradiction.
        If $s,t\in A_2\setminus A_1$, then $x$ is in $P_t$ and $P_t'$ is the corresponding $M_\cC^e$-alternating $v$--$t$ path in $G-e$.
        By a similar argument, one can deduce a contradiction as $|\Delta(P_t')\cap \rho(A_1,A_2)|>1$.
        This completes the proof.
	\end{proof}
	
    Now we can deduce the torsoid in $G$ corresponds to $\cP(x)$ in $D_\cC^e$.
	\begin{theorem}\label{ssm-torsoid}
        Any $\cP(x)$ in $D_\cC^e$ corresponds to a $C_4$-torsoid $\cT_x=(H,\varepsilon)$ in $G$ with $V(H) = \{v_1,v_2,v_3,v_4\}$ such that:
		\begin{itemize}
			\item $v_1:=\{x^+\}$,
            \item $v_2:=\{x^-\}$,
            \item $v_3:= \begin{cases}
                \bigcup_{x_i \in I_p^x} \{x_i^+, x_i^-\} \cup \{{i^x}^+\} & \text{ if } i^x \neq s,\\
                \bigcup_{x_i \in I_p^x} \{x_i^+, x_i^-\} \cup \{s\} & \text{ if } i^x = s,
            \end{cases}$
            \item $v_4:= \begin{cases}
                \bigcup_{x_i \in O_p^x} \{x_i^+, x_i^-\} \cup \{{o^x}^-\} & \text{ if } o^x \neq t,\\
                \bigcup_{x_i \in O_p^x} \{x_i^+, x_i^-\} \cup \{t\} & \text{ if } o^x = t,
            \end{cases}$
            \item $\varepsilon(v_1v_2)=\emptyset$,
			\item $\varepsilon(v_2v_3)=\bigcup_{x_i\in I^x}\{x_i^+, x_i^-\}$,
            \item $\varepsilon(v_1v_4)=\bigcup_{x_i\in O^x}\{x_i^+, x_i^-\}$,
            \item $\varepsilon(v_3v_4)=
            \begin{cases}
            \bigcup_{x_i\in R^x}\{x_i^+, x_i^-\} \cup \{{i^x}^-, {o^x}^+, s, t\} & \text{ if } i^x \neq s \text{ and } o^x \neq t,\\
            \bigcup_{x_i\in R^x}\{x_i^+, x_i^-\} \cup \{{o^x}^+, t\} & \text{ if } i^x = s \text{ and } o^x \neq t,\\
            \bigcup_{x_i\in R^x}\{x_i^+, x_i^-\} \cup \{{i^x}^-, s\} & \text{ if } i^x \neq s \text{ and } o^x = t,\\
            \bigcup_{x_i\in R^x}\{x_i^+, x_i^-\} & \text{ if } i^x = s \text{ and } o^x = t.
            \end{cases}
            $
		\end{itemize}
        Furthermore, this is a one-to-one correspondence.
	\end{theorem}
    \begin{proof}
        By definition, $\cT_x$ satisfies (T1), (T3) and (T4) in \Cref{def-torsoid}.
        Then \Cref{prop:acyclic_ordering,obs:PX_properties} ensure that there exists neither an edge with tail $x$ and head in $I_p^x \cup \{i^x\}$ nor an edge with head $x$ and tail in $O_p^x \cup \{o^x\}$, which implies that (T6) holds.
        Moreover, (T7) holds, since there are only $C_4$-torsoids in $G$ by~\Cref{lem:decisive,lem-decisive-even}.
        

        Now we show that (T2) holds. It's trivial that $v_1$ and $v_2$ are tight sets.
        \Cref{prop:inoutsep} ensures that $(V(D_\cC^e) \setminus I_p^x,\{i^x\} \cup I_p^x)$ is a good directed $1$-separation in $D_\cC^e$.
        Note that $t \in V(D) \setminus (I_p^x \cup \{i^x\})$.
        Then
        \[
        \phi(V(D_\cC^e) \setminus I_p^x,\{i^x\} \cup I_p^x) = 
        \begin{cases}
        ((V(D_\cC^e) \setminus (I_p^x \cup \{s,t\})) \cup \{v_e\}, \{i^x\} \cup I_p^x) & \text{ if } i^x \neq s,\\
        ((V(D_\cC^e) \setminus (I_p^x \cup \{s,t\})) \cup \{v_e\}, I_p^x \cup \{v_e\}) & \text{ if } i^x = s,
        \end{cases}
        \]
        which is a directed $1$-separation in $D(G,M_\cC^e)$ by~\cref{prop:surmap}.
        Then, by~\cref{prop:cut-side2},
        \[
        \sigma(\{i^x\} \cup I_p^x) =
        \begin{cases}
        
        \bigcup_{x_i \in I_p^x} \{x_i^+,x_i^-\} \cup \{{i^x}^+\} & \text{ if }  i^x \neq s,\\
        \bigcup_{x_i \in I_p^x} \{x_i^+,x_i^-\} \cup \{s\} & \text{ if }  i^x = s,
        \end{cases}
        \]
        which implies $\sigma(\{i^x\} \cup I_p^x) = v_3$.
        Thus, since $\sigma(\{i^x\} \cup I_p^x)$ is tight, $v_3$ is a tight set in $G$. Up to symmetry, $v_4=\sigma(O_p^x \cup \{o^x\})$ is also tight.
%

        Next, we will show that (T5) holds. \Cref{def-partition-ssm} ensures that none of $E(x,I^x)$, $E(x,O^x)$, $E(\{x\}\cup I^x,I_p^x \cup \{i^x\})$ and $E(\{x\}\cup O^x,O_p^x \cup \{o^x\})$ is empty. This implies that none of $E(v_1,v_4\cup \edgefkt{v_1v_4})$, $E(v_4,v_1\cup \edgefkt{v_1v_4})$, $E(v_2,v_3\cup \edgefkt{v_2v_3})$ and $E(v_3,v_2\cup \edgefkt{v_2v_3})$ is empty.
        Then we show that both $E(v_3,v_4\cup \edgefkt{v_3v_4})$ and $E(v_4,v_3\cup \edgefkt{v_3v_4})$ are not empty. Up to symmetry, it suffices to show $E(v_3,v_4\cup \edgefkt{v_3v_4})\neq \emptyset$.
        If $i^x =s$, then $st\in E(v_3,v_4\cup \edgefkt{v_3v_4})$.
        If $i^x \neq s$, then $i^x$ corresponds to an edge between $v_3$ and $\edgefkt{v_3v_4}$.
        Then by \Cref{prop:surmap,tightsettoseparation}, for every $v_iv_j\in E(H)$, there is an edge from $v_i\cup \edgefkt{v_iv_j}$ to $v_j$.

        \begin{claim}
			$\varepsilon(v_2v_3)$ is maximal passable set between $v_2$ and $v_3$, and $\varepsilon(v_1v_4)$ is maximal passable set between $v_1$ and $v_4$.
		\end{claim}
		\begin{claimproof}
			By symmetry, it suffices to show the first case. By \Cref{prop:inoutsep}, $x$ is the out-1-separator of $\{x\}\cup I^x$ and $i^x$ is the in-1-separator of $\{i^x\}\cup I_p^x$ in $D_\cC^e$. By definition of $i^x$, $I^x$ and $I_p^x$, it's easy to check that $i^x$ is also the in-1-separator of $\{i^x\}\cup I^x\cup I_p^x$.
            Then $\bigcup_{x_i\in I^x}\{x_i^+, x_i^-\}=\varepsilon(v_2v_3)$ is a passable set between $v_2$ and $v_3$ by \Cref{tightsettoseparation,prop:surmap}.
            
            Suppose that there exists a larger passable set $\edgefkt{v_2v_3}'$ in $v_2 \cup \edgefkt{v_2v_3} \cup v_3$ between $v_2$ and $v_3$. Then $\edgefkt{v_2v_3}'\cap v_3\neq \emptyset$ and $Y:=v_2\cup \edgefkt{v_2v_3}'$ is a tight set.
            By \Cref{tightsettoseparation}, $Y$ corresponds to a directed 1-separation $(A,B)$ in $D(G,M_\cC^e)$ with $x$ being the out-separator of $A$. Since $x\neq v_e$, there exists precisely one preimage of $(A,B)$ under $\phi$ in $D_\cC^e$, say $(A',B')$. Note that $A$ is the side corresponding to $Y$ as $x$ is the out-separator of $A$, then by the construction of $\phi$, we can get that $x$ is the out-separator of $A'$ and $\sigma(A')=Y$.
            Since $Y\cap v_3\neq \emptyset$, $A'\cap (I_p^x\cup \{i^x\})\neq \emptyset$. By \Cref{near-partition}, $I^x\cap (I_p^x\cup \{i^x\})=\emptyset$, which implies that there exists a directed $v$--$t$ path $P$ for any $v\in A'\cap (I_p^x\cup \{i^x\})$ avoiding $x$ by the definition of $I^x$. Then $P$ contributes an edge from $A'\setminus B'$ to $B'\setminus A'$, a contradiction to \Cref{di-separation}.
            Thus $\varepsilon(v_2v_3)$ is maximal passable set between $v_2$ and $v_3$.
		\end{claimproof}
		
		\begin{claim}
			$\varepsilon(v_3v_4)$ is maximal passable set between $v_3$ and $v_4$.
		\end{claim}
		\begin{claimproof}
            First, we show that there exists a good directed 1-separation $(A,B)$ with $B:=R^x \cup I_p^x \cup \{i^x,o^x\}$ and $A\cap B=\{o^x\}$.
            Let $uv$ be an edge with $v \in B\setminus A$.
            Then $u\notin I^x$, since $x$ is the out-separator of $I^x \cup \{x\}$, and $u\notin \{x\} \cup O^x \cup O_p^x$, since $o^x$ is the out-separator of $\{x\} \cup O^x \cup O_p^x \cup \{o^x\}$ by~\Cref{prop:inoutsep}. Then $u\in B$, which implies that $(A,B)$ is a good directed 1-separation in $D_\cC^e$.
            By symmetry, there exists another good directed 1-separation $(A',B')$ in $D_\cC^e$ with $A':=R^x \cup O_P^x \cup \{i^x,o^x\}$ and $A' \cap B' = \{i^x\}$.
            By \Cref{prop:surmap,tightsettoseparation}, $\sigma(B)$ and $\sigma(A')$ are tight sets. Note that $\sigma(B)=\edgefkt{v_3v_4}\cup v_3$ and $\sigma(A')=\edgefkt{v_3v_4}\cup v_4$ by \Cref{prop:cut-side2}. This implies that $\varepsilon(v_3v_4)$ is a passable set between $v_3$ and $v_4$.
            
            Next, suppose for a contradiction that there exists a larger passable set $\edgefkt{v_3v_4}'$ between $v_3$ and $v_4$ in $v_3 \cup \edgefkt{v_3v_4} \cup v_4$. Then $\edgefkt{v_3v_4}'\cap (v_3\cup v_4)\neq \emptyset$. Assume without loss of generality that $\edgefkt{v_3v_4}'\cap v_4\neq \emptyset$ and let $Y:=\edgefkt{v_3v_4}'\cup v_3$. Then $Y$ corresponds to a directed 1-separation $(A_1,B_1)$ in $D(G,M_\cC^e)$ by \Cref{tightsettoseparation}. Let $z$ be the separator of $(A_1,B_1)$. Let $(A_1',B_1')$ be an arbitrary preimage of $(A_1,B_1)$ under $\phi$ in $D_\cC^e$, and let $z'$ be the separator of $(A_1',B_1')$.
            
            Now, we will show that $\sigma(B_1')=Y$. Suppose for a contradiction that $\sigma(A_1')=Y$.
            Then $i^x \in A_1'$ and $x \in B_1' \setminus A_1'$. 
            Since there is no edge from $A_1'\setminus \{z'\}$ to $B_1'\setminus \{z'\}$,
            every $i^x$--$x$~path has to contain $z'$.
            Thus the definition of $i^x$ implies that $z'=i^x$.
            We show that the unique edge $f \in \rho(A_1',B_1') \cap M_\cC^e$, i.e.\ the edge that corresponds to $z$, has both endpoints in $\sigma(A_1)$, a contradiction.
            
            If $z \neq v_e$, then $z'=i^x\notin \{s,t\}$ by definition of $\phi$.
            Thus ${i^x}^+$ and ${i^x}^-$ are the endpoints of $f$.
            Furthermore, ${i^x}^+\in v_3$ and ${i^x}^-\in \edgefkt{v_3v_4}$, by definition of $v_3$ and $\edgefkt{v_3v_4}$.
            This implies that ${i^x}^+$ and ${i^x}^-$ are contained in $\sigma(A_1')$.
            If $z=v_e$, then $z'=i^x \in \{s,t\}$, which implies that $z'=i^x=s$.
            In particular, $s \in A_1' \cap B_1'$.
            By \cref{obs:s_and_t}, $t\in A_1'$, which implies $t\in \sigma(A_1')$.
            Thus in both cases, the endpoints of $f$ are contained in $\sigma(A_1')$, a contradiction.
            Therefore $\sigma(B_1')=Y$.

            Since $Y\cap v_4\neq \emptyset$ and $\sigma(B_1')=Y$, we have $B_1'\cap O_p^x\neq \emptyset$.
            We show $z'\in B_1'\cap O_p^x$. Suppose not, and let $y\in B_1'\cap O_p^x$ be arbitrary. By definition of $O_p^x$, $y$ is an internal vertex of some directed $x$--$o^x$ path $P$. Then $P$ contributes an edge from $A_1'$ to $B_1'$ that avoids $z'$, a contradiction to \Cref{di-separation}.
            If $z \neq v_e$, then $(A_1',B_1')$ is unique and $z=z'\in B_1'\cap O_p^x$.
            Note that $t \in B_1'$ since $s \in B_1'$.
            Since $z$ is the separator, there is no directed $x$--$t$ path in $D_\cC^e - z$. The definition of $o^x$ ensures that there is also no directed $x$--$t$ path in $D_\cC^e - o^x$, which implies that $z$ lies on every directed $o^x$--$t$ path. However, $z\in O_p^x$, so by \Cref{obs:PX_properties}-\Cref{obs:PX2}, there exists a directed $z$--$o^x$ path. This leads to a contradiction, because $D_\cC^e$ is acyclic.
            If $z=v_e$, then $(A_1',B_1')$ is one of the $((A_1\setminus \{v_e\})\cup \{s,t\}, (B_1\setminus \{v_e\})\cup \{s\})$ and $((A_1\setminus \{v_e\})\cup \{t\}, (B_1\setminus \{v_e\})\cup \{s,t\})$. This shows that $z'\in \{s,t\}$.
            Note that neither $s$ nor $t$ is in $O_p^x$, a contradiction to $z'\in O_p^x$.
            

            
            This shows that $\edgefkt{v_3v_4}'\cap v_4=\emptyset$.
            Thus $\varepsilon(v_3v_4)$ is maximal passable set between $v_3$ and $v_4$.
		\end{claimproof}
        Therefore, $\cT_x=(H,\varepsilon)$ is a torsoid in $G$.

        We now prove the furthermore assertion. Note that any edge $f_i$ in $M_\cC^e$ represents the flag $F_i$ of $\cC$, it means that any torso based on $\cC$ corresponds to some vertex in $D_\cC^e$ except the source and sink.
        Since $G$ is path-like and without nontrivial even $2$-separations, there are only $C_4$-torsoids in $G$ by \Cref{lem:decisive,lem-decisive-even}. In other words, any torsoid $\cT_x$ in $G$ is cleaved by a unique torso $S_x$ in $\mathcal{C}$. While $S_x$ corresponds to vertex $x$ in $D_\cC^e$. Thus the torsoid $\cT_x$ corresponds precisely to the $\cP(x)$ of $D_\cC^e$. Here we are done.
    \end{proof}
    
    \begin{proposition}\label{alternating_path_in_torsoid}
        Let $u\in V(G)\setminus \{s,t\}$, then the following properties hold for $\cT_x$:
        \begin{description}
            \item[If $i^x \neq s$:] A vertex $u$ is contained in $v_3\cup \edgefkt{v_2v_3}$ if and only if $u$ is an internal vertex of some $M_\cC^e$-alternating ${i^x}^-{i^x}^+$--$x^-$~path that avoids $e$ and $x^-x^+$.
            \item[If $i^x = s$:] A vertex $u$ is contained in $v_3\cup \edgefkt{v_2v_3}$ if and only if $u$ is an internal vertex of some $M_\cC^e$-alternating $s$--$x^-$~path that avoids $e$ and $x^-x^+$.
            \item[If $o^x \neq s$:] A vertex $u$ is contained in $v_4\cup \edgefkt{v_1v_4}$ if and only if $u$ is an internal vertex of some $M_\cC^e$-alternating $x^+$--${o^x}^-{o^x}^+$~path that avoids $e$ and $x^-x^+$.
            \item[If $o^x = s$:] A vertex $u$ is contained in $v_4\cup \edgefkt{v_1v_4}$ if and only if $u$ is an internal vertex of some $M_\cC^e$-alternating $x^+$--$t$~path that avoids $e$ and $x^-x^+$.
        \end{description}
    \end{proposition}
    \begin{proof}
        We first show the cases for $v_3\cup \edgefkt{v_2v_3}$. Let $u\in v_3\cup \edgefkt{v_2v_3}$.
        By \Cref{ssm-torsoid}, $u$ corresponds to a vertex $v_u$ in $D_\cC^e$ such that $v_u\in I^x\cup I_p^x\cup \{i^x\}$. \Cref{obs:PX_properties}-\Cref{obs:PX1} ensures that $v_u$ is in some directed $i^x$--$x$~path, which implies that $u$ is in some $M_\cC^e$-alternating path $P$ that avoids $e$ and $x^-x^+$ by \Cref{alternating_to_dipath} and the definition of $i^x$.
        If $i^x=s$, then since $v_2\cup \edgefkt{v_2v_3}$ is tight, $P$ starts at $s$ and ends at $x^-$.
        Otherwise, $P$ starts at ${i^x}^-$ and ends at $x^-$ by \Cref{prop:cut-side2}. For both cases, $u$ is an internal vertex of $P$.
        Conversely, if $u$ is an internal vertex of some desired $M_\cC^e$-alternating path, then $v_u$ is in some directed $i^x$--$x$~path by \Cref{alternating_to_dipath,prop:cut-side2}. This implies that $v_u\in I^x\cup I_p^x\cup \{i^x\}$ by \Cref{obs:PX_properties}-\Cref{obs:PX1}. Thus $u\in v_3\cup \edgefkt{v_2v_3}$.
        By a similar argument, we can deduce the case for $v_4\cup \edgefkt{v_1v_4}$.
    \end{proof}

    

	\subsection{Interaction of torsoids}
	
	\begin{lemma}\label{interaction-ssm}
		Let $a=x_i$ and $b=x_j$ be two distinct vertices in $D_\cC^e$ with $1\le i<j\le n$. Let $\{\{a\},I^a,I^a_p, \{i^a\}, R^a, \{o^a\}, O^a_p, O^a\}$ and $\{\{b\},I^b,I^b_p, \{i^b\}, R^b,\{o^b\}, O^b_p, O^b\}$ be $\cP(a)$ and $\cP(b)$ of $D_\cC^e$, respectively. Then $\cP(a)$ and $\cP(b)$ interact in the following ways:
		\begin{enumerate}
			\item If $a$ and $b$ are incomparable, then $I^b \cup O^b \cup \{b\} \subseteq R^a$ and $I^a \cup O^a \cup \{a\} \subseteq R^b$.
			Furthermore, both $(I^b_p \cup \{i^b\})\cap (O^a_p \cup \{o^a\})$ and $(I^a_p \cup \{i^a\}) \cap (O^b_p \cup \{o^b\})$ are empty.
			\item If $a$ and $b$ are comparable, then $a \in I^b \cup I^b_p \cup \{i^b\} \cup R^b$ and $b \in O^a \cup O^a_p \cup \{o^a\} \cup R^a$.
			\begin{enumerate}[label=(\alph*)]
        		\item\label{itm:inter_a} If $a \in I^b$, then $I^a \cup O^a_p \cup O^a \subseteq I^b$, $o^a \in I^b \cup \{b\}$, $I^a_p \subseteq I^b \cup I^b_p$, $i^a \in I^b \cup I^b_p \cup \{i^b\}$, $b \in R^a \cup \{o^a\}$ and $R^b \cup O^b_p \cup \{o^b\} \cup O^b \subseteq R^a$.
                \item\label{itm:inter_b} If $a \in I^b_p$, then $I^a \cup I^a_p \subseteq I^b_p$ and $i^a \in I^b_p \cup \{i^b\}$.
                \item\label{itm:inter_c} If $a \in R^b \cup \{i^b\}$, then $I^a \cup I^a_p \cup \{i^a\} \subseteq R^b$.
        		\item\label{itm:inter_d} If $b \in O^a$, then $O^b \cup I^b \cup I^b_p \subseteq O^a$, $i^b \in O^a \cup \{a\}$, $O^b_p \subseteq O^a \cup O^a_p$, $o^b \in O^a \cup O^a_p \cup \{o^a\}$, $a \in R^b \cup \{i^b\}$ and $R^a \cup I^a_p \cup \{i^a\} \cup I^a \subseteq R^b$.
        		\item\label{itm:inter_e} If $b \in O^a_p$, then $O^b \cup O^b_p \subseteq O^a_p$ and $o^b \in O^a_p \cup \{o^a\}$.
        		\item\label{itm:inter_f} If $b \in R^a \cup \{o^a\}$, then $O^b \cup O^b_p \cup \{o^b\} \subseteq R^a$.
            \end{enumerate}
		\end{enumerate}
	\end{lemma}
	Note that $a \in R^b \cup \{i^b\}$ only if $s\in R^b$ and $b \in R^a \cup \{o^a\}$ only if $t\in R^a$ in the setting of~\cref{interaction-ssm}.

	\begin{proof}
		First assume that $a$ and $b$ are incomparable in $D_\cC^e$.
		By \Cref{def-partition-ssm}, every directed $s$--$t$~path, that contains an element of $I^b \cup O^b \cup \{b\}$, has to contain $b$.
		Furthermore, for every element $x \in I^a_p \cup O^a_p \cup \{a\}$, there exists a directed $s$--$t$~path that contains $x$ and $a$.
		We can deduce that $I^b \cup O^b \cup \{b\} \subseteq R^a$.
		Note that every vertex in $I^b_p$ reaches $b$ and every vertex in $O^a_p$ can be reached from $a$.
		Thus $I^b_p \cap O^a_p = \emptyset$.
		By symmetry, that is by interchanging the roles of $s$ and $t$ and of $\cP(a)$ and $\cP(b)$, we can deduce $I^a \cup O^a \cup \{a\} \subseteq R^b$.
        Similarly, $I^a_p \cap O^b_p = \emptyset$.
		
		Now we assume that $a$ and $b$ are comparable. Then \Cref{prop:acyclic_ordering} ensures that there exists a directed $a$--$b$~path as $i<j$.
        Thus $a \notin \{o^b\} \cup O^b_p \cup O^b$ and $b\notin \{i^a\}\cup I^a\cup I^a_p$. Up to symmetry, that is up to interchanging the roles of $s$ and $t$ and of $\cP(a)$ and $\cP(b)$, it suffices to prove the first three cases.

        \begin{description}
            \item[\cref{itm:inter_a}]
                If $a \in I^b$, then there is no directed $a$--$t$~path in $D_\cC^e - b$, which implies that there exists a (possibly trivial) directed $o^a$--$b$~path by definition of $o^a$.
                Thus $b \in R^a \cup \{o^a\}$ by \cref{obs:PX_properties}.
                Furthermore, the $o^a$--$b$~path implies that every directed $a$--$o^a$~path extends to some directed $a$--$b$~path. Then by \cref{obs:PX_properties}-\Cref{obs:PX2}, every vertex of $O^a \cup O^a_p$ is an internal vertex of a directed $a$--$b$~path. Thus $O^a \cup O^a_p \subseteq I^b$ and $o^a \in I^b \cup \{b\}$ by definition of $I^b$.
                By definition of $I^a$, every path from some vertex in $I^a$ to $t$ contains $a$, and thus contains also $b$ since $a \in I^b$, which implies $I^a \subseteq I^b$.
                By definition of $I^b$ and $i^b$, every directed $s$--$a$~path has to contain $i^b$. This implies that there exists a (possibly trivial) directed $i^b$--$i^a$~path by definition of $i^a$.
                Thus every directed $i^a$--$a$~path extends to a directed $i^b$--$b$~path.
                Then, by \cref{obs:PX_properties}-\Cref{obs:PX1}, $I^a_p \subseteq I^b \cup I^b_p$ and $i^a \in I^b \cup I^b_p \cup \{i^b\}$.
                Furthermore, \Cref{near-partition} ensures that $R^b \cup O^b_p \cup \{o^b\} \cup O^b \subseteq R^a$.
            \item[\cref{itm:inter_b}]
                If $a \in I^b_p$, then in $D_\cC^e - i^b$, there is a directed $a$--$b$~path but no directed $s$--$a$~path.
                This implies that there exists a (possibly trivial) directed $i^b$--$i^a$~path.
                Thus every directed $i^a$--$a$~path extends to a directed $i^b$--$b$~path.
                By \cref{obs:PX_properties}-\Cref{obs:PX1}, every vertex of $I^a \cup I^a_p$ is an internal vertex of some directed $i^b$--$b$~path, which implies that $I^a \cup I^a_p \subseteq I^b \cup I^b_p$.
                Since $a \in I^b_p$, there is a directed $a$--$t$~path avoiding $b$. Then any path from $I^a \cup I^a_p\cup \{i^a\}$ to $a$ can extend to a path to $t$ avoiding $b$. By the definition of $I^b$, we can deduce that $I^a \cup I^a_p \subseteq I^b_p$ and $i^a \in I^b_p \cup \{i^b\}$.
            \item[\cref{itm:inter_c}]
                If $a \in R^b \cup \{i^b\}$, then $s\in R^b$, since there exists a directed $a$--$b$~path. Extend this path, we can get a directed $s$--$b$~path $P$ that contains $a$.
                Note that $P$ has to contain $i^b$ by the definition of $i^b$. \cref{obs:PX_properties}-\Cref{obs:PX1} ensures that $a$ is an internal vertex of $Pi^b$, which implies that there exists a (possibly trivial) directed $a$--$i^b$~path.
                Then any path from $I^a \cup I^a_p\cup \{i^a\}$ to $a$ can extend to a path to $i^b$. Thus we can deduce that $I^a \cup I^a_p \cup \{i^a\} \subseteq R^b$. \qedhere
        \end{description}
	\end{proof}

    By \Cref{interaction-ssm,ssm-torsoid}, we can deduce the interactions between any two distinct torsoids in $G$.
	
	\begin{corollary}\label{interaction-ssm2}
		Let $\cT=(H, \edgefkt{})$ and $\cT'=(H', \edgefkt{}')$ be torsoids in $G$, where $V(H)=\{v_1,v_2,v_3,v_4\}$ and $V(H')=\{v_1',v_2',v_3',v_4'\}$ as defined in~\Cref{ssm-torsoid}. Let $v_1=\{x^+\}$, $v_2=\{x^-\}$, $v_1'=\{y^+\}$ and $v_2'=\{y^-\}$.
        Then $\cT$ and $\cT'$ interact in the following ways:
		\begin{enumerate}
			\item If there does not exist an $M_\cC^e$-alternating $s$--$t$~path containing the edges $x^+x^-$ and $y^+y^-$ in $G$, then $\edgefkt{v_1v_4} \cup v_1 \cup \edgefkt{v_1v_2} \cup v_2 \cup \edgefkt{v_2v_3} \subseteq \edgefkt{}'(v_3'v_4')$ and $\edgefkt{}'(v_1'v_4') \cup v_1' \cup \edgefkt{}'(v_1'v_2') \cup v_2' \cup \edgefkt{}'(v_2'v_3') \subseteq \edgefkt{v_3v_4}$.
            Furthermore, both $v_3 \cap v_4'$ and $v_4 \cap v_3'$ are empty.
			\item If there exists an $M_\cC^e$-alternating $s$--$t$~path in which the edge $x^+x^-$ precedes the edge $y^+y^-$ in $G$, then $v_1, v_2 \subseteq \edgefkt{}'(v_2'v_3') \cup v_3' \cup \edgefkt{}'(v_3'v_4')$ and $v_1', v_2' \subseteq \edgefkt{v_1v_4} \cup v_4 \cup \edgefkt{v_3v_4}$.
			\begin{enumerate}[label=(\alph*)]
        		\item\label{itm:cor_inter_a} If $v_2 \subseteq \edgefkt{}'(v_2'v_3')$,
                then $\edgefkt{v_1v_4} \cup v_1 \cup \edgefkt{v_1 v_2} \cup v_2 \cup \edgefkt{v_2v_3} \subseteq \edgefkt{}'(v_2'v_3')$, $v_4 \subseteq v_2' \cup \edgefkt{}'(v_2'v_3')$, $v_3 \subseteq \edgefkt{}'(v_2' v_3') \cup v_3'$, $v_1' \subseteq \edgefkt{v_3v_4}$ and $\edgefkt{}'(v_3'v_4') \cup v_4' \cup \edgefkt{}'(v_4'v_1')\subseteq \edgefkt{v_3v_4}$.
                \item\label{itm:cor_inter_b} If $v_2 \subseteq v_3'$, then $v_1\cup \edgefkt{v_2v_3} \cup v_3 \subseteq v_3'$.
                \item\label{itm:cor_inter_c} If $v_2 \subseteq \edgefkt{}'(v_3'v_4')$, then $\edgefkt{v_2v_3} \cup v_3 \subseteq \edgefkt{}'(v_3'v_4')$.
        		\item\label{itm:cor_inter_d} If $v_1' \subseteq \edgefkt{v_1v_4}$,
                then $\edgefkt{}'(v_2'v_3') \cup v_2' \cup \edgefkt{}'(v_1' v_2') \cup v_1' \cup \edgefkt{}'(v_1'v_4') \subseteq \edgefkt{v_1v_4}$, $v_3' \subseteq v_1 \cup \edgefkt{v_1v_4}$, $v_4' \subseteq \edgefkt{v_1 v_4} \cup v_4$, $v_2 \subseteq \edgefkt{}'(v_3'v_4')$ and $\edgefkt{v_3v_4} \cup v_3 \cup \edgefkt{v_3v_2}\subseteq \edgefkt{}'(v_3'v_4')$.
        		\item\label{itm:cor_inter_e} If $v_1' \subseteq v_4$, then $v_2'\cup \edgefkt{}'(v_1'v_4') \cup v_4' \subseteq v_4$.
        		\item\label{itm:cor_inter_f} If $v_1' \subseteq \edgefkt{v_3v_4}$, then $\edgefkt{}'(v_1'v_4') \cup v_4' \subseteq \edgefkt{v_3v_4}$.
            \end{enumerate}
		\end{enumerate}
	\end{corollary}
    Note that $v_2 \subseteq \edgefkt{}'(v_3'v_4')$ only if $s \in \edgefkt{v_3'v_4'}$ and $v_1' \subseteq \edgefkt{v_3v_4}$ only if $t \in \edgefkt{v_3v_4}$ in the statement of~\cref{interaction-ssm2}.
    
    \section{Graphs without distinguished edge}\label{sec:torsoidsG}
		In this section, we analyze the structure of torsoids in path-like graphs that do not contain nontrivial even $2$-separation.
    For graphs containing distinguished edges, all torsoids and their interactions can be derived using the source-sink model described in \Cref{sec:sourcesink}.
    Thus, it remains to examine the subclass of graphs without distinguished edges.

	
	Let $G$ be an arbitrary path-like graph without nontrivial even $2$-separation and without distinguished edges, and let $\cC$ be a path-like maximal family of nested tight cuts in $G$.
    Moreover, let $G^+$ be the graph obtained from $G$ by adding an edge $e$ between a lower vertex of $\cF_1$ and an upper vertex of $\cF_n$.
	We note that $e$ is a distinguished edge of $G^+$ and $G^+$ does not contain a nontrivial even $2$-separation.
    Note that every tight set in $G^+$ is also a tight set in $G$ by construction of $G^+$.
    
    Let $M_\cC^e$ the the perfect matching of $G^+$ as defined in~\cref{sec:sourcesink}.
    The proof of~\cref{prop:unique_perfect_matching} shows:
    \begin{proposition}\label{prop:unique_perfect_matching_2}
        Let $G$ be a path-like graph without nontrivial even $2$-separation and without distinguished edge.
		Then $M_{\cC}^e$ is the unique perfect matching of $G^+$ containing $e$. \qed
	\end{proposition}
    Given a tight cut $C$ in $G$, let $C^+$ be the corresponding cut in $G^+$, then either $C^+=C$ or $C^+=C\cup \{e\}$.
    By \Cref{prop:unique_perfect_matching_2}, every perfect matching of $G^+$ except $M_\cC^e$ intersects $C^+$ precisely once.
    Let $\cC^+:=\{C^+: C \in \cC\}$.
    Since $M_{\cC}^e=\{f_i: i \in [n]\} \cup \{e\}$ and by the choice of $\cC$, $M_\cC^e$ intersects every cut of $\cC^+$ precisely once. This implies that every cut in $\cC^+$ is tight in $G^+$.
    Thus $\cC^+$ is a path-like maximal family of nested tight cuts in $G^+$.
	This implies, in particular, that $G^+$ is path-like.
	Then, by~\Cref{sec:sourcesink}, the structure of torsoids in $G^+$ is well-understood.

    \subsection{Torsoids}
	
	Let $\cT = (H, \edgefkt{})$ be an arbitrary torsoid in $G^+$ with $V(H) := \{v_1, v_2, v_3, v_4\}$ as in \Cref{ssm-torsoid}.
	Corresponding to $\cT$ we construct a tuple $\cR_\cT = (I,\edgefkt{})$ as follows:
	Let $I$ be a $C_4$ with vertices $w_1, \cdots, w_4$ and edges $w_1w_2, \cdots, w_4w_1$.
	We choose $w_1:= v_1$, $w_2:= v_2$, $w_3 \subseteq v_3$ and $w_4 \subseteq v_4$.
	More precisely, let
	\begin{itemize}
		\item $\edgefkt{w_1w_2}:= \edgefkt{v_1v_2} = \emptyset$;
		\item $\edgefkt{w_3w_4} := \edgefkt{v_3v_4}$;
		\item $\edgefkt{w_2w_3}$ be $\subseteq$-maximal with $\edgefkt{v_2v_3} \subseteq \edgefkt{w_2w_3} \subseteq \edgefkt{v_2v_3} \cup v_3$ such that either $\edgefkt{w_2w_3} = \edgefkt{v_2v_3}$ or
		\begin{itemize}
			\item $\boundary{w_2 \cup \edgefkt{w_2w_3}}$ contains precisely three edges of $M_\cC^e$,
			\item $s \in \edgefkt{w_2w_3}$, and
			\item any edge in $\boundary{w_2 \cup \edgefkt{w_2w_3}} \setminus \{e\}$ has one endpoint in $w_2 \cup \edgefkt{w_2w_3}$ as an upper vertex of some flag.
		\end{itemize}
		We set $w_3:= v_3 \setminus \edgefkt{w_2w_3}$;
		\item $\edgefkt{w_1w_4}$ be $\subseteq$-maximal with $\edgefkt{v_1v_4} \subseteq \edgefkt{w_1w_4} \subseteq \edgefkt{v_1v_4} \cup v_4$ such that either $\edgefkt{w_1w_4} = \edgefkt{v_1v_4}$ or
		\begin{itemize}
			\item $\boundary{w_1 \cup \edgefkt{w_1w_4}}$ contains precisely three edges of $M_\cC^e$,
			\item $t \in \edgefkt{w_1w_4}$, and
			\item any edge in $\boundary{w_1 \cup \edgefkt{w_1w_4}} \setminus \{e\}$ has one endpoint in $w_1 \cup \edgefkt{w_1w_4}$ as a lower vertex of some flag.
		\end{itemize}
		We set $w_4:= v_4 \setminus \edgefkt{w_1w_4}$.
	\end{itemize}
	
	We will show:
	\begin{theorem}\label{lem:torsoids_in_G}
		The tuple $\cR_\cT = (I,\edgefkt{})$ is a torsoid in $G$.
	\end{theorem}
	We note that $\cR_\cT$ and $\cT$ are cleaved by the same torso of $\cC$, since each vertex set $w_i$ of $\cR_\cT$ is contained in the vertex set $v_i$ of $\cT$.
	Given an arbitrary torsoid $\cR$ in $G$, let $\cS$ be some torso of $\cC$ such that $\cS$ cleaves $\cR$.
	Then there exists a torsoid $\cT$ in $G^+$ that is also cleaved by $\cS$, which implies $\cR = \cR_\cT$ by \Cref{rem:cyclic_torsoids}-\Cref{cleave-one}.
	This shows that every torsoid in $G$ is of the form $\cR_\cT$ for some torsoid $\cT$ of $G^+$, which provides the desired structure of torsoids in $G$.
	
	Before we prove~\Cref{lem:torsoids_in_G}, we need the following two lemmata.
	
	\begin{lemma}\label{lem:comparing_tight_cut_G}
        Let $C$ be a cut in $G$.
        Then $C$ is tight in $G$ and $C^+$ is not tight in $G^+$ if and only if $C^+$ meets the following three conditions:
        \begin{enumerate}[label=(\alph*)]
            \item\label{itm:comparing_a} $C^+$ contains precisely three edges of $M_\cC^e$,
			\item\label{itm:comparing_b} $s$ and $t$ are on different sides of $C^+$, and
			\item\label{itm:comparing_c} for every edge $f \in C^+ \setminus \{e\}$, the endpoint of $f$ in the side containing $s$ is an upper vertex of some flag.
        \end{enumerate}
	\end{lemma}
\begin{proof}
		Let $C$ be a tight cut in $G$ such that $C^+$ is not tight in $G^+$.
		Since $M_\cC^e$ is the unique perfect matching in $G^+$ that is not a perfect matching in $G$, there are at least two edges of $M_\cC^e$ in $C^+$.
		Since $C$ is tight in $G$, both sides of $C^+$ are odd and thus $M_\cC^e$ intersects $C^+$ oddly.
		This implies $|M_\cC^e \cap C^+| \geq 3$.
		Let $P_1$ and $P_2$ be internally disjoint $M_\cC^e$-alternating $s$--$t$~paths in $G$, which exist by~\Cref{lem:two_disjoint_paths}.
        For each $i\in[2]$, note that $\Delta(P_i)=(E(P_i) \setminus (M_\cC^e\setminus \{e\})) \cup ((M_\cC^e\setminus \{e\}) \setminus E(P_i))$ is a perfect matching of $G$ by \cref{prop:alternating_s_t_paths}. Then $|\Delta(P_i)\cap C|=1$ as $C$ is tight in $G$, which implies that $|\Delta(P_i)\cap ((M_\cC^e \setminus \{e\}) \cap C)| = | ((M_\cC^e \setminus \{e\}) \setminus E(P_i)) \cap C | \le 1$. Thus $|M_\cC^e \cap C| \leq 2$, which implies $|M_\cC^e \cap C^+| = 3$ and $e\in C^+$, i.e.\ $s$ and $t$ are on different sides of $C^+$. This shows that $C^+$ meets both conditions \cref{itm:comparing_a} and \cref{itm:comparing_b}.
        Let $j,k \in \N$ such that $M_\cC^e \cap C^+=\{e,f_k,f_j\}$.

        Let $M$ be an arbitrary perfect matching of $G$, and let $P$ be the $M_\cC^e$-alternating $s$--$t$~path with $\Delta(P)=M$, which exists by~\Cref{prop:alternating_s_t_paths}.
        Since $C$ is tight in $G$, $M$ intersects $C$ in precisely one edge $g$.
        Then either $g \in \{f_k,f_j\}$ or $g \in E(G) \setminus M_\cC^e$.
        In the former case, $P$ intersects $C$ precisely in the edge of $\{f_k,f_j\} \setminus \{g\}$.
        In the latter case, $P$ intersects $C$ precisely in the edges $f_k, f_j$ and $g$.
        Since $G$ is matching covered, there exists a perfect matching $M_k$ in $G$ such that $f_k\in M_k$.
        Let $P_k$ be an $M_\cC^e$-alternating $s$--$t$~path such that $\Delta(P_k)=M_k$.
        This implies that $P_k$ intersects $C^+$ precisely in $f_j$. \Cref{remark:alternating_path} ensures that the endpoint of $f_j$ that lies on the same side as $s$ is an upper vertex of $\cF_j$.
        Similarly, we can deduce that the endpoint of $f_k$ that lies on the same side as $s$ is an upper vertex of $\cF_k$.
        Meanwhile, for every edge $f\in C^+ \setminus M_\cC^e$, there also exists a perfect matching $N$ in $G$ such that $f\in N$. Let $P_N$ be the $M_\cC^e$-alternating $s$--$t$~path in $G$ such that $\Delta(P_N)=N$. Then $P_N$ intersects $C^+$ in the three edges $f,f_k,f_j$. Since both endpoints of $f_k$ and $f_j$ that lie on the same side as $s$ are upper vertices,
        $P_N$ traverses $f_k$ and $f_j$ from the side of $s$ to the side of $t$ by \Cref{remark:alternating_path}.
        Thus $P_N$ traverses $f$ from the side of $t$ to the side of $s$.
        Then \Cref{remark:alternating_path} implies that the endpoint of $f$ that lies on the side of $s$ is also an upper vertex of some flag.
        This shows that $C^+$ meets condition \cref{itm:comparing_c}.

		Conversely, assume that $C^+$ meets the conditions \cref{itm:comparing_a}, \cref{itm:comparing_b}, and \cref{itm:comparing_c}. Then $M_\cC^e$ witnesses that $C^+$ is not tight in $G^+$ and that both sides of $C^+$ are odd. Note that $C^+= C\cup \{e\}$.
        Let $M$ be an arbitrary perfect matching in $G$. Then there exists an $M_\cC^e$-alternating $s$--$t$~path $P'$ in $G$ with $\Delta(P')=M$ by \Cref{prop:alternating_s_t_paths}.
        By condition \cref{itm:comparing_c} and \Cref{remark:alternating_path}, $P'$ intersects $C^+$ in edges of $M_\cC^e$ if and only if it moves from the side of $s$ to the side of $t$. Then the first intersection of $P'$ with $C^+$ is an edge of $M_\cC^e$.
        Furthermore, if $P'$ intersects $C^+$ more than once, then the second intersection is not an edge of $M_\cC^e$ but the third.
        Since $|C^+\cap M_\cC^e|=3$ and $e\in C^+$, either $P'$ intersects $C$ precisely in one edge in $M_\cC^e\setminus \{e\}$, or $P'$ intersects $C$ in two edges in $C\cap M_\cC^e$ and one edge in $E(P')\setminus M_\cC^e$.
        Then the corresponding perfect matching $\Delta(P')=(E(P') \setminus (M_\cC^e\setminus \{e\})) \cup ((M_\cC^e\setminus \{e\}) \setminus E(P'))$ intersects $C$ precisely once.
        By \Cref{prop:alternating_s_t_paths}, the cut $C$ intersects any perfect matching of $G$ precisely once, which means $C$ is a tight cut in $G$. This completes the proof.
	\end{proof}
    By~\cref{lem:comparing_tight_cut_G}, $w_2 \cup \edgefkt{w_2w_3}$ and $w_1 \cup \edgefkt{w_1w_4}$ are tight sets in $G$.
	
	\begin{corollary}\label{T2forG}
	    $w_3$ and $w_4$ are tight sets in $G$.
	\end{corollary}
    \begin{proof}
        It's trivial for $w_3=v_3$ and $w_4=v_4$ by \Cref{ssm-torsoid}, as $v_3$ and $v_4$ are tight sets in $G^+$, then also in $G$.
        Thus we assume $w_3\subset v_3$ and $w_4\subset v_4$.
        By \Cref{lem:comparing_tight_cut_G} and construction of $\cR_\cT$, $v_2\cup \edgefkt{w_2w_3}$ is a tight set in $G$.
        Since $w_3=v_3\setminus \edgefkt{w_2w_3}$, $w_3$ is an odd set.
        Moreover, any edge in $\partial(w_3)$ is either in $\partial(v_2\cup \edgefkt{w_2w_3})$ or in $\partial(v_3)$.
        Then any perfect matching in $G$ intersects $\partial(w_3)$ precisely once, which implies that $w_3$ is tight.
        By a similar argument, we can deduce that $w_4$ is also tight in $G$.
    \end{proof}
    
	\begin{lemma}\label{lem:maximal_passable}
        $\edgefkt{v_3v_4}$ is the maximal subset of $w_3 \cup \edgefkt{v_3v_4} \cup w_4$ that is passable between $w_3$ and $w_4$ in $G$.
	\end{lemma}
	\begin{proof}
        For $w_3=v_3$ and $w_4=v_4$, $\edgefkt{v_3v_4}$ is passable between $w_3$ and $w_4$ in $G$ by \Cref{ssm-torsoid}.
        Without loss of generality we assume $w_3\subset v_3$. By the choice of $\edgefkt{w_2w_3}$, $v_2 \cup \edgefkt{w_2w_3}$ is a tight set in $G$.
        We first show that $Y:=v_4 \cup \edgefkt{v_1v_4} \cup v_1 \cup \edgefkt{v_1 v_2} \cup v_2 \cup \edgefkt{w_2w_3}$ is tight in $G$.
        If $t\in Y$, then any edge in $\partial(Y)$ is either in $\partial(v_2 \cup \edgefkt{w_2w_3})$ or in $\partial(v_4)$. Thus any perfect matching in $G$ intersects $\partial(Y)$ precisely once, which implies that $Y$ is tight.
        Otherwise, $\partial(Y)$ meets the condition in \Cref{lem:comparing_tight_cut_G}, which means that $Y$ is a tight set.
        Therefore, the complement of $Y$, i.e.\ $w_3 \cup \edgefkt{v_3v_4}$, is also tight in $G$.
        Similarly, when $w_4\subset v_4$, $w_4 \cup \edgefkt{v_3v_4}$ is tight in $G$. 
        This implies that $\edgefkt{v_3v_4}$ is passable between $w_3$ and $w_4$ whenever $w_3\subset v_3$ or $w_4\subset v_4$.
        
        
		Suppose that there exists a larger passable set $Z$ between $w_3$ and $w_4$ in $w_3 \cup \edgefkt{v_3v_4} \cup w_4$, then $Z \cap (w_3\cup w_4) \neq \emptyset$, and $w_4 \cup Z$ and $w_3 \cup Z$ are both tight in $G$.
        Since $(w_4 \cup Z) \cap v_4 = w_4$ and $(w_3 \cup Z) \cap v_3 = w_3$ are odd, by \Cref{crossing_tight_sets}, $v_4 \cup Z = (w_4 \cup Z) \cup v_4$ and $v_3 \cup Z = (w_3 \cup Z) \cup v_3$ are also tight in $G$.
        As $\edgefkt{v_3v_4}$ is maximal passable between $v_3$ and $v_4$ in $G^+$, one of $v_3 \cup Z$ and $v_4 \cup Z$ is not tight in $G^+$. Without loss of generality, we assume that $v_4 \cup Z$ is not tight in $G^+$. 
        Then precisely one of $s$ and $t$ is in $v_4 \cup Z$ by \Cref{lem:comparing_tight_cut_G}. 
        Furthermore, since $t \in v_4 \cup \edgefkt{v_3v_4} \subseteq v_4 \cup Z$, for every edge $f \in \boundary{v_4 \cup Z}$, the endpoint of $f$ in $v_4 \cup Z$ is a lower vertex of some flag.
        We will show that there exists an edge $g$ in $\boundary{v_4 \cup Z}$ contradicting this property.
        
		Let $x^+, x^- \in V(G)$ such that $v_1 = \{x^+\}$ and $v_2 = \{x^-\}$, and let $f\in E(G) \setminus M_\cC^e$ be an edge incident with $x^+$.
        Note that $x^+x^-\in M_\cC^e\setminus\{e\}$.
        Since $G$ is matching covered and by~\cref{prop:alternating_s_t_paths}, there exists an $M_\cC^e$-alternating $s$--$t$~path $P$ that contains $f$. Then $P$ also contains $x^+x^-$.
        Note that $v_1 \cup \edgefkt{v_1 v_4}$ is a tight set in $G^+$ by \Cref{ssm-torsoid}, thus also a tight set in $G$. This implies that $M_\cC^e \cap \boundary{v_1 \cup \edgefkt{v_1 v_4}}=\{x^+x^-\}$ and $|\Delta(P)\cap \boundary{v_1 \cup \edgefkt{v_1 v_4}}|=1$. Let $g$ be this unique edge. Since every edge in $\boundary{v_1 \cup \edgefkt{v_1 v_4}}$ is either $x^+x^-$ or in $\boundary{v_4}$, $g\in \boundary{v_4} \cap \boundary{v_1 \cup \edgefkt{v_1 v_4}}$, which means $g\in \boundary{v_4 \cup Z}$. Moreover, $g\in E(P)\setminus M_\cC^e$ by definition of $\Delta(P)$.
        By \Cref{remark:alternating_path} and since $x^+$ is the lower vertex of $x^+x^-$, $P$
        traverses $x^+x^-$ from $x^-$ to $x^+$, and thus enters $v_1 \cup \edgefkt{v_1v_4}$ through $x^+x^-$.
        Therefore $P$ leaves $v_1 \cup \edgefkt{v_1v_4}$ through $g$.
        By \Cref{remark:alternating_path}, $g$ contributes an endpoint in $v_4 \cup Z$ as an upper vertex of some flag, a contradiction.
        
        Therefore, $\edgefkt{v_3v_4}$ is maximal passable set between $w_3$ and $w_4$.
	\end{proof}
	
	\begin{proof}[Proof of~\Cref{lem:torsoids_in_G}]
		By definition, the tuple $\cR_\cT=(H, \edgefkt{})$ satisfies (T1), (T3) and (T4) in~\Cref{def-torsoid}.
        \Cref{T2forG} ensures that (T2) holds, and \Cref{lem:decisive,lem-decisive-even} ensure that (T7) holds, as there are only $C_4$-torsoids in $G$. Furthermore, by \Cref{ssm-torsoid}, $E(v_1,v_3)$ and $E(v_2,v_4)$ are both empty in $G^+$, then also in $G$, which implies that $E(w_1,w_3)$ and $E(w_2,w_4)$ are empty in $G$, i.e.\ (T6) holds. Therefore, it remains to prove that $\cR_\cT$ satisfies (T5).

        By \Cref{ssm-torsoid}, $E(v_2, v_3\cup \edgefkt{v_2v_3})\neq \emptyset$, then $E(w_2, w_3\cup \edgefkt{w_2w_3})\neq \emptyset$ by definition of $\cR_\cT$.
        Moreover, $E(v_2\cup \edgefkt{v_2v_3},v_3)\neq \emptyset$, then $E(w_2\cup \edgefkt{w_2w_3},w_3)\neq \emptyset$ when $w_3=v_3$. If $w_3\subset v_3$, then there exists an edge between $w_3$ and $\edgefkt{w_2w_3}\setminus \edgefkt{v_2v_3}$ as $v_3$ is a tight set. Thus $E(w_2\cup \edgefkt{(w_2w_3)},w_3)\neq \emptyset$ also holds. Similarly, $E(w_1, w_4\cup \edgefkt{w_1w_4})\neq \emptyset$ and $E(w_1\cup \edgefkt{(w_1w_4)},w_4)\neq \emptyset$.
        
        Now, we show that the set $\edgefkt{w_2w_3}$ is maximal passable between $w_2$ and $w_3$ in $w_2 \cup \edgefkt{w_2w_3}\cup w_3$.
        It's trivial when $w_3=v_3$, therefore we assume $w_3\subset v_3$.
        Since $w_2 \cup \edgefkt{w_2w_3}$ is tight in $G$ by~\Cref{lem:comparing_tight_cut_G} and $w_3 \cup \edgefkt{w_2w_3}=v_3 \cup \edgefkt{v_2v_3}$ is also tight in $G$ by \Cref{ssm-torsoid}, $\edgefkt{w_2w_3}$ is passable between $w_2$ and $w_3$.
        Suppose for a contradiction that there exists a larger passable set $Z$ between $w_2$ and $w_3$ in $w_2 \cup \edgefkt{w_2w_3}\cup w_3$, then $Z\cap w_3\neq \emptyset$ as $|w_2|=1$. Since $\edgefkt{v_2v_3}$ is the maximal passable set between $v_2$ and $v_3$ in $v_2 \cup \edgefkt{v_2v_3}\cup v_3$, $w_2\cup Z=v_2\cup Z$ cannot be tight in $G^+$. Then $\partial(w_2\cup Z)$ meets conditions in \Cref{lem:comparing_tight_cut_G}. This implies that $Z$ meets the conditions for the choice of $\edgefkt{w_2w_3}$, a contradiction to the maximality of $\edgefkt{w_2w_3}$.
        By the symmetric arguments, $\edgefkt{w_1w_4}$ is maximal passable between $w_1$ and $w_4$ in $w_1 \cup \edgefkt{w_1w_4}\cup w_4$.

        Next, we show the property of (T5) holds for $w_3\cup \edgefkt{w_3w_4}\cup w_4$.
        Let $x^+, x^- \in V(G)$ such that $w_1 = \{x^+\}$ and $w_2 = \{x^-\}$. As $G$ is matching covered, there is a perfect matching $N$ in $G$ that contains the edge $x^+x^-$. 
        Note that any edge in $\partial(w_3)$ is either in $\partial(w_2\cup \edgefkt{w_2w_3})$ or in $\partial(w_4\cup \edgefkt{w_3w_4})$.
        Since $w_2 \cup \edgefkt{w_2w_3}$ is tight in $G$, $N\cap \partial(w_2\cup \edgefkt{w_2w_3})=\{x^+x^-\}$.
        This implies that $N\cap \partial(w_3) \subset \partial(w_4\cup \edgefkt{w_3w_4})$ as $w_3$ is also tight in $G$.
        Thus, there exists an edge between $w_3$ and $w_4\cup \edgefkt{w_3w_4}$.
        By symmetry, there exists an edge between $w_4$ and $w_3\cup \edgefkt{w_3w_4}$.
        Moreover, by \Cref{lem:maximal_passable}, $\edgefkt{w_3w_4}$ is maximal passable between $w_3$ and $w_4$ in $w_3 \cup \edgefkt{w_3w_4} \cup w_4$.

        This shows that $\cR_\cT$ satisfies (T5), i.e. $\cR_\cT$ is a torsoid in $G$.
	\end{proof}

	\subsection{Interaction of torsoids}
    Let $G$ be a path-like graph without nontrivial even $2$-separations, and let $\cC$ be a path-like maximal family of nested tight cuts in $G$.
    Let $\cR_\cT=(I, \edgefkt{})$ and $\cR_{\cT'}=(I', \edgefkt{}')$ be distinct torsoids in $G$, where $V(I)=\{w_1,w_2,w_3,w_4\}$ and $V(I')=\{w_1',w_2',w_3',w_4'\}$.
    Let $w_1=\{x^+\}$, $w_2=\{x^-\}$, $w_1'=\{y^+\}$ and $w_2'=\{y^-\}$.
    
	\begin{lemma}\label{lem:interaction_torsoids_in_G}
        $\cR_{\cT}$ and $\cR_{\cT'}$ interact in the following ways:
        		\begin{enumerate}
			\item If there does not exist an $M_\cC^e$-alternating $s$--$t$~path containing the edges $x^+x^-$ and $y^+y^-$ in $G$, then $\edgefkt{w_1w_4} \cup w_1 \cup \edgefkt{w_1w_2} \cup w_2 \cup \edgefkt{w_2w_3} \subseteq \edgefkt{}'(w_3'w_4')$ and $\edgefkt{}'(w_1'w_4') \cup w_1' \cup \edgefkt{}'(w_1'w_2') \cup w_2' \cup \edgefkt{}'(w_2'w_3') \subseteq \edgefkt{w_3w_4}$.
            Furthermore, both $w_3 \cap w_4'$ and $w_4 \cap w_3'$ are empty.
			\item If there exists an $M_\cC^e$-alternating $s$--$t$~path in which the edge $x^+x^-$ precedes the edge $y^+y^-$ in $G$, then $w_1, w_2 \subseteq \edgefkt{}'(w_2'w_3') \cup w_3' \cup \edgefkt{}'(w_3'w_4')$ and $w_1', w_2' \subseteq \edgefkt{w_1w_4} \cup w_4 \cup \edgefkt{w_3w_4}$.
                Furthermore,
                \[
                \edgefkt{w_2w_3} \cup w_3 \subseteq \begin{cases}
                \edgefkt{}'(w_2'w_3') \cup w_3' & \text{ if } w_2 \subseteq \edgefkt{}'(w_2'w_3'),\\
                w_3' & \text{ if } w_2 \subseteq w_3',\\
                \edgefkt{}'(w_3'w_4') & \text{ if } w_2 \subseteq \edgefkt{}'(w_3'w_4'),
                \end{cases}
                \]
                 and
                \[
                \edgefkt{}'(w_1'w_4') \cup w_4' \subseteq \begin{cases}
                \edgefkt{w_1w_4} \cup w_4 & \text{ if } w_1' \subseteq \edgefkt{w_1w_4},\\
                    w_4 & \text{ if } w_1' \subseteq w_4,\\
                    \edgefkt{w_3w_4} & \text{ if } w_1' \subseteq \edgefkt{w_3w_4}.
                \end{cases}
                \]
		\end{enumerate}
	\end{lemma}

    Note that $w_2 \subseteq \edgefkt{}'(w_3'w_4')$ only if $s\in \edgefkt{}'(w_3'w_4')$ and $w_1' \subseteq \edgefkt{w_3w_4}$ only if $t\in \edgefkt{w_3w_4}$ in the statement of~\cref{lem:interaction_torsoids_in_G}.

    Before we can prove~\cref{lem:interaction_torsoids_in_G}, we need the following two propositions.

    \begin{proposition}\label{prop:interaction_helper}
        If $G$ has no distinguished edge, $\edgefkt{}'(v_2'v_3') \subsetneq \edgefkt{}'(w_2'w_3')$, $w_2 \nsubseteq w_2' \cup \edgefkt{}'(w_2'w_3')$ and there does not exist an $M_\cC^e$-alternating $s$--$t$~path in which $y^+y^-$ precedes $x^+x^-$, then $\edgefkt{}'(w_2'w_3') \cap (\edgefkt{w_2w_3} \cup w_3) = \emptyset$.
    \end{proposition}
    \begin{proof}
        By the definition of $\cR_{\cT'}$, we have that $s \in \edgefkt{}'(w_2'w_3')$, $\boundary{w_2' \cup \edgefkt{}'(w_2'w_3')}$ contains precisely three edges of $M_\cC^e$ and for every edge $f \in \boundary{w_2' \cup \edgefkt{w_2'w_3'}} \setminus \{e\}$, the endpoint of $f$ in $w_2' \cup \edgefkt{w_2'w_3'}$ is an upper vertex of some flag since $\edgefkt{}'(v_2'v_3') \subsetneq \edgefkt{}'(w_2'w_3')$. 
        Note that both $e$ and $y^+y^-$ are in $\boundary{w_2' \cup \edgefkt{}'(w_2'w_3')} \cap M_\cC^e$.
        Let $f_i=x_i^+x_i^-$ be the third edge in $\boundary{w_2' \cup \edgefkt{}'(w_2'w_3')} \cap M_\cC^e$, then the upper vertex $x_i^- \in w_2' \cup \edgefkt{}'(w_2'w_3')$.

        \begin{claim}\label{clm:helper1}
            Every $M_\cC^e$-alternating $s$--$x^+x^-$~path in $G$ contains $f_i$ and avoids $y^+y^-$.
        \end{claim}
        \begin{claimproof}
            Let $P$ be an arbitrary $M_\cC^e$-alternating $s$--$x^+x^-$~path in $G$.
            Since $w_2 \nsubseteq w_2' \cup \edgefkt{}'(w_2'w_3')$, both $x^+$ and $x^-$ are not in $w_2' \cup \edgefkt{}'(w_2'w_3')$. Note that $s \in \edgefkt{}'(w_2'w_3')$, then \Cref{remark:alternating_path} ensures that $P$ interacts $\boundary{w_2' \cup \edgefkt{}'(w_2'w_3')}\cap (M_\cC^e\setminus \{e\})$.
            Since there is no $M_\cC^e$-alternating $s$--$t$~path in which $y^+y^-$ precedes $x^+x^-$, $P$ avoids $y^+y^-$.
            This implies that $P$ contains $f_i$.
        \end{claimproof}
        \begin{claim}\label{clm:helper2}
            Every $M_\cC^e$-alternating $x_i^+$--$x^+x^-$~path in $G-f_i$ avoids $w_2' \cup \edgefkt{}'(w_2'w_3')$.
        \end{claim}
        \begin{claimproof}
            Let $Q$ be an arbitrary $M_\cC^e$-alternating $x_i^+$--$x^+x^-$~path in $G-f_i$.
            Furthermore, let $P'$ be some $M_\cC^e$-alternating $s$--$f_i$~path.
            Then $P'$ and $Q$ concatenate to an $M_\cC^e$-alternating $s$--$x^+x^-$~path.
            By~\cref{clm:helper1} and \Cref{remark:alternating_path}, the concatenation of $P'$ and $Q$ leaves $w_2' \cup \edgefkt{}'(w_2'w_3')$ only through $f_i$.
            Thus we can deduce that $Q$ does not intersect the set $w_2' \cup \edgefkt{}'(w_2'w_3')$ since its endvertices are not contained in $w_2' \cup \edgefkt{}'(w_2'w_3')$.
        \end{claimproof}
        
        \Cref{clm:helper1} implies that there is no $M_\cC^e$-alternating $s$--$x^+x^-$~path in $G-f_i$.
        Then by the definition of $i^x$ in \Cref{def-partition-ssm}, \Cref{obs:PX_properties,alternating_to_dipath,alternating_path_in_torsoid}, any $M_\cC^e$-alternating $s$--$f_i$~path in $G$ contains ${i^x}^-{i^x}^+$, and $i^x\neq s$.

        By \Cref{alternating_path_in_torsoid}, any vertex $u$ in $\edgefkt{v_2v_3} \cup v_3$ is in some $M_\cC^e$-alternating ${i^x}^-{i^x}^+$--$x^-$~path in $G$ that avoids $x^+x^-$, which implies that $u$ is in some $M_\cC^e$-alternating $x_i^+$--$x^+x^-$~path.
        Furthermore, \Cref{clm:helper2} shows that $\edgefkt{v_2v_3} \cup v_3$ does not intersect $w_2' \cup \edgefkt{w_2'w_3'}$.
        Thus $\edgefkt{w_2w_3} \cup w_3$  does not intersect $w_2' \cup \edgefkt{w_2'w_3'}$ by the construction of $\cR_\cT$.
    \end{proof}

    \begin{proposition}\label{prop:interaction_helper_2}
        If $G$ has no distinguished edge, and there does not exist an $M_\cC^e$-alternating $s$--$t$~path which contains $y^+y^-$ and $x^+x^-$, then $\edgefkt{}'(w_2'w_3') \subseteq \edgefkt{w_3w_4}$ when $\edgefkt{}'(v_2'v_3') \subsetneq \edgefkt{}'(w_2'w_3')$, and $\edgefkt{}'(w_1'w_4') \subseteq \edgefkt{w_3w_4}$ when $\edgefkt{}'(v_1'v_4') \subsetneq \edgefkt{}'(w_1'w_4')$.
    \end{proposition}
    \begin{proof}
        We first assume that $\edgefkt{}'(v_2'v_3') \subsetneq \edgefkt{}'(w_2'w_3')$ and show that $\edgefkt{}'(w_2'w_3') \subseteq \edgefkt{w_3w_4}$.
        
        Since there does not exist an $M_\cC^e$-alternating $s$--$t$~path which contains $y^+y^-$ and $x^+x^-$, $w_2 \nsubseteq w_2' \cup \edgefkt{}'(w_2'w_3')$ by \Cref{obs:PX_properties}-\Cref{obs:PX1}, \Cref{alternating_to_dipath,ssm-torsoid}.
        Thus we can apply~\cref{prop:interaction_helper}, which shows that $\edgefkt{}'(w_2'w_3') \subseteq \edgefkt{w_3w_4} \cup w_4 \cup \edgefkt{w_1w_4} \cup w_1 \cup \edgefkt{w_1w_2} \cup w_2$.

        By \Cref{alternating_path_in_torsoid}, every vertex in $\edgefkt{}'(w_2'w_3')$ is in some $M_\cC^e$-alternating $s$--$y^+y^-$ path $P$, and every vertex in $w_4 \cup \edgefkt{w_1w_4} \cup w_1 \cup \edgefkt{w_1w_2} \cup w_2$ is in some $M_\cC^e$-alternating $x^+x^-$--$t$ path $P'$.
        Since there does not exist an $M_\cC^e$-alternating $s$--$t$~path which contains both $y^+y^-$ and $x^+x^-$, $V(P)\cap V(P')=\emptyset$. This implies that the set $\edgefkt{}'(w_2'w_3')$ does not intersect $w_4 \cup \edgefkt{w_1w_4} \cup w_1 \cup \edgefkt{w_1w_2} \cup w_2$.
        Thus $\edgefkt{}'(w_2'w_3') \subseteq \edgefkt{w_3w_4}$.


        By symmetry, we can deduce that $\edgefkt{}'(v_1'v_4') \subsetneq \edgefkt{}'(w_1'w_4')$ implies that $\edgefkt{}'(w_1'w_4') \subseteq \edgefkt{w_3w_4}$.
    \end{proof}

	\begin{proof}[Proof of~\cref{lem:interaction_torsoids_in_G}]
        If $G$ contains a distinguished edge, then~\cref{interaction-ssm2} implies the statement of~\cref{lem:interaction_torsoids_in_G}.
        Thus we can assume that $G$ does not contain a distinguished edge.
	    We apply \cref{interaction-ssm2} to $\cT$ and $\cT'$ in $G^+$ and deduce that $w_1, w_2 \subset \edgefkt{}'(w_2'w_3') \cup w_3' \cup \edgefkt{}'(w_3'w_4')$ by the definition of $\cR_\cT$ and $\cR_{\cT'}$.
        Note that there exists an $M_\cC^e$-alternating $s$--$t$~path containing the edges $x^+x^-$ and $y^+y^-$ in $G$ if and only if such a path exists in $G^+$.
        
        If there does not exist an $M_\cC^e$-alternating $s$--$t$~path containing the edges $x^+x^-$ and $y^+y^-$, then $w_1 \cup \edgefkt{w_1w_2} \cup w_2 \subseteq \edgefkt{}'(w_3'w_4')$, $ w_1' \cup \edgefkt{}'(w_1'w_2') \cup w_2'\subseteq \edgefkt{w_3w_4}$ and, both, $w_3 \cap w_4'$ and $w_4 \cap w_3'$ are empty, by \cref{interaction-ssm2} and definition of $\cR_{\cT}$ and $\cR_{\cT'}$.
        
        If $\edgefkt{}'(w_1'w_4')=\edgefkt{}'(v_1'v_4')$, then $\edgefkt{}'(w_1'w_4') \subseteq \edgefkt{}(w_3w_4)$ by \cref{interaction-ssm2} and definition of $\cR_\cT$ and $\cR_{\cT'}$.
        Otherwise, $\edgefkt{}'(v_1'v_4') \subsetneq \edgefkt{}'(w_1'w_4')$ and~\cref{prop:interaction_helper_2} implies $\edgefkt{}'(w_1'w_4') \subseteq \edgefkt{}(w_3w_4)$.
        By the same argument, $\edgefkt{}'(w_2'w_3') \subseteq \edgefkt{}(w_3w_4)$ and
        $\edgefkt{w_1w_4} \cup \edgefkt{w_2w_3} \subseteq \edgefkt{}'(w_3'w_4')$.

        Now we assume that there exists an $M_\cC^e$-alternating $s$--$t$~path in which the edge $x^+x^-$ precedes the edge $y^+y^-$.
        Then $w_1, w_2 \subset \edgefkt{}'(w_2'w_3') \cup w_3' \cup \edgefkt{}'(w_3'w_4')$ and $w_1', w_2' \subset \edgefkt{w_1w_4} \cup w_4 \cup \edgefkt{w_3w_4}$ by \Cref{alternating_path_in_torsoid} and the definition of $\cR_\cT$ and $\cR_{\cT'}$.
        Up to symmetry, it suffices to prove the first three cases.
        \begin{description}
            \item[If $w_2 \subset \edgefkt{}'(w_2'w_3')$:] Then $v_2 \subset \edgefkt{}'(v_2'v_3') \cup v_3'$.
            Thus \cref{interaction-ssm2} implies $\edgefkt{w_2w_3} \cup w_3 \subseteq \edgefkt{}'(w_2' w_3') \cup w_3'$.
            \item[If $w_2 \subset w_3'$:] Then either $w_3'=v_3'$ or $\edgefkt{}'(v_2'v_3') \subsetneq \edgefkt{}'(w_2'w_3')$.
            In the former case, \cref{interaction-ssm2} implies $\edgefkt{w_2w_3} \cup w_3 \subseteq w_3'$.
            In the latter case, we apply \cref{prop:interaction_helper}.
            Since there is an $M_\cC^e$-alternating $s$--$t$~path in which $x^+x^-$ precedes $y^+y^-$, by \Cref{ob_path-like}-\Cref{obs:order_of_path} and \Cref{ob_flag}, there does not exist an $M_\cC^e$-alternating $s$--$t$~path in which $y^+y^-$ precedes $x^+x^-$.
            Thus \cref{prop:interaction_helper} shows that $(\edgefkt{w_2w_3} \cup w_3) \cap \edgefkt{}'(w_2'w_3') = \emptyset$.
            By \cref{interaction-ssm2} and the definition of $\cS_\cT$ and $\cS_{\cT'}$, $\edgefkt{w_2w_3} \cup w_3 \subseteq w_3' \cup \edgefkt{}'(w_2'w_3')$ and thus $\edgefkt{w_2w_3} \cup w_3 \subseteq w_3'$.
            \item[If $w_2 \subset \edgefkt{}'(w_3'w_4')$:] Then $\edgefkt{w_2w_3} \cup w_3 \subseteq \edgefkt{}'(w_3'w_4')$ by \cref{interaction-ssm2} and the definition of $\cR_\cT$ and $\cR_{\cT'}$. \qedhere
        \end{description}
	\end{proof}

	\section*{Acknowledgements}
	
	The second author gratefully acknowledges support by a doctoral scholarship of the Studienstiftung des deutschen Volkes.
	The third author gratefully acknowledges support by China Scholarship Council (No.~202206770031) and Sino-German (CSC-DAAD) Postdoc Scholarship Program, 2022 (57607866).\\
	All authors contributed equally to this work.
	
	\bibliographystyle{alpha}
	\bibliography{reference}
	
\end{document}